\pdfoutput=1
\RequirePackage[l2tabu, orthodox]{nag}

\documentclass[reqno]{amsart}

\usepackage{lmodern}
\usepackage[T1]{fontenc}
\usepackage[utf8]{inputenc}
\usepackage[english]{babel}
\usepackage{microtype} 

\usepackage{amsmath,amssymb,amsthm,mathrsfs,latexsym,mathtools,mathdots,booktabs,enumerate,tikz,bm,url,comment}
\usepackage[enableskew,vcentermath]{youngtab}
\usepackage[centertableaux]{ytableau}

\usetikzlibrary{arrows,positioning}
\usepackage[capitalize,noabbrev]{cleveref}

\usepackage{todonotes}
\presetkeys%
    {todonotes}%
    {inline,backgroundcolor=yellow}{}

\usepackage{graphicx}

\usepackage{ifpdf}
\graphicspath{{./}{./figures/}}
\usepackage{epstopdf}
\DeclareGraphicsExtensions{.png,.jpg,.eps,.epsf,.pdf}

\newtheorem{theorem}{Theorem}
\newtheorem{proposition}[theorem]{Proposition}
\newtheorem{lemma}[theorem]{Lemma}
\newtheorem{corollary}[theorem]{Corollary}
\newtheorem{conjecture}[theorem]{Conjecture}

\theoremstyle{definition}
\newtheorem{definition}[theorem]{Definition}
\newtheorem{example}[theorem]{Example}
\newtheorem{remark}[theorem]{Remark}
\newtheorem{question}[theorem]{Question}

\newtheorem{openproblem}[theorem]{Open Problem}

\newcommand{\defin}[1]{\emph{#1}}

\newcommand{\setN}{\mathbb{N}}

\newcommand{\setQ}{\mathbb{Q}}

\newcommand{\setC}{\mathbb{C}}

\newcommand{\avec}{\mathbf{a}}
\newcommand{\bvec}{\mathbf{b}}
\newcommand{\cvec}{\mathbf{c}}

\newcommand{\xvec}{\mathbf{x}}
\newcommand{\yvec}{\mathbf{y}}

\newcommand{\wvec}{\mathbf{w}}

\newcommand{\svec}{\mathbf{s}}

\newcommand{\nuvec}{{\boldsymbol\nu}}

\newcommand{\symS}{S}

\newcommand{\psumP}{\mathrm{p}}
\newcommand{\elementaryE}{\mathrm{e}}
\newcommand{\macdonaldH}{\mathrm{\tilde H}}
\newcommand{\schurS}{\mathrm{s}}
\newcommand{\LLT}{\mathrm{G}}
\newcommand{\chrom}{\mathrm{X}}

\newcommand{\SSYT}{\mathrm{SSYT}}
\newcommand{\SYT}{\mathrm{SYT}}
\newcommand{\PF}{\mathrm{PF}}

\DeclareMathOperator{\length}{\ell}

\DeclareMathOperator{\inv}{inv}
\DeclareMathOperator{\asc}{asc}

\DeclareMathOperator{\comaj}{comaj}
\DeclareMathOperator{\halfsinks}{halfsinks}
\DeclareMathOperator{\halfsources}{halfsources}
\DeclareMathOperator{\sinks}{sinks}
\DeclareMathOperator{\sources}{sources}
\DeclareMathOperator{\Sinks}{Sinks}
\DeclareMathOperator{\area}{area}

\DeclareMathOperator{\cocharge}{cc}
\DeclareMathOperator{\sort}{sort}
\DeclareMathOperator{\dinv}{dinv}
\DeclareMathOperator{\ides}{ides}

\title[LLT polynomials and graphs with cycles]{LLT polynomials, chromatic quasisymmetric functions and graphs with cycles}

\author{Per Alexandersson and Greta Panova}
\address{Dept. of Mathematics, Royal Institute of Technology, SE-100 44 Stockholm, Sweden}
\address{Dept. of Mathematics, University of Pennsylvania. Philadelphia, PA}
\email{per.w.alexandersson@gmail.com}
\email{panova@math.upenn.edu}

\begin{document}

\begin{abstract}
We use a Dyck path model for unit-interval graphs to study 
the chromatic quasisymmetric functions introduced by Shareshian and Wachs,
as well as vertical strip --- in particular, unicellular LLT polynomials.

We show that there are parallel phenomena regarding $\elementaryE$-positivity of these two families of polynomials.
In particular, we give several examples where the LLT polynomials behave like a ``mirror image'' of the 
chromatic quasisymmetric counterpart.

The Dyck path model is also extended to circular arc digraphs to obtain larger families of polynomials.
This circular extensions of LLT polynomials has not been studied before.
A lot of the combinatorics regarding unit interval graphs carries over to this more general setting,
and we prove several statements regarding the $\elementaryE$-coefficients of 
chromatic quasisymmetric functions and LLT polynomials.

In particular, we believe that certain $\elementaryE$-positivity conjectures hold in all these families above.
Furthermore, we study vertical-strip LLT polynomials, for which there is no natural chromatic quasisymmetric counterpart.
These polynomials are essentially modified Hall--Littlewood polynomials, and are therefore of special interest.

In this more general framework, we are able to give a natural combinatorial interpretation for 
the $\elementaryE$-coefficients for the line graph and the cycle graph, in both the chromatic and the LLT setting.
\end{abstract}

\maketitle
\setcounter{tocdepth}{1}
\tableofcontents

\section{Introduction}

In \cite{Stanley95Chromatic}, Stanley introduced a generalization of the chromatic polynomial for graphs,
called the \emph{chromatic symmetric function}, given as a sum over all proper colorings of the graph.
Shareshian and Wachs introduced a refinement of chromatic symmetric functions in \cite{ShareshianWachs2014},
depending on an extra parameter $q$, called the \emph{chromatic quasisymmetric functions}.
For unit interval graphs, the corresponding functions turn out to be \emph{symmetric} and related to the representation theory of Hessenberg varieties.

Without restricting to proper colorings, the chromatic quasisymmetric
function by Stanley is trivially $e_1^n$ on any graph with $n$ vertices.
However, by allowing \emph{all colorings} together with the $q$-parameter keeping track of the ascend statistic, we recover the unicellular LLT polynomials,
a subfamily of the polynomials defined by Lascoux, Leclerc and Thibon in \cite{Lascoux97ribbontableaux}.
LLT polynomials have received a lot of attention recently due to their close connection with 
modified Macdonald polynomials and diagonal harmonics, see \emph{e.g.} \cite{Haglund2005Macdonald,CarlssonMellit2015}.

\bigskip
We say that a function $f(\xvec;q)$ in $Sym_{\setQ}[q]$ is $\elementaryE$-\defin{positive} if
the coefficients $c_\mu(q)$ in the expansion
\[
 f(\xvec;q) = \sum_\mu c_\mu(q) \elementaryE_\mu(\xvec)
\]
are polynomials with non-negative coefficients. 
The main open problem regarding chromatic symmetric functions is 
to show that given a $(3+1)$-avoiding poset $P$, the chromatic symmetric 
function of the incomparability graph of $P$, is $\elementaryE$-positive.
From the works of \cite{GuayPaquet2013}, it can be shown that it suffices to prove 
the conjecture of $\elementaryE$-positivity for $(3+1)$-avoiding posets, 
with the additional assumption that the poset is also $(2+2)$-avoiding.

The incomparability graphs of such posets can be realized as \emph{natural unit interval graphs}.
The number of natural unit interval graphs on $n$ vertices is known to be enumerated by the Catalan numbers, see \cite[Exercise 6.19]{StanleyEC2}.
In this paper, we describe a model indexed by Dyck paths that naturally realizes these incomparability graphs.
Our model is closely related to the model used in \cite{GuayPaquet2016,Haglund2012compositionalShuffle}
and we borrow some terminology from the world of parking functions.
We also apply this model to ribbon LLT polynomials.

\medskip 

Here are the highlights of the paper:
\begin{itemize}
\item We extend the family of natural unit interval graphs to \emph{circular arc digraphs} and show 
that many properties of corresponding chromatic quasisymmetric functions and LLT polynomials
can be extended to this setting. Note that circular arc digraphs are 
in general \emph{not} incomparability graphs of posets, but the related graphs are still claw free. 
Circular arc digraphs have independently been considered in~\cite{Ellzey2016}. The undericted circular graphs were considered in \cite{Stanley95Chromatic} as ``circular indifference graphs''.

\item We pose an analogue of the Stanley--Stembridge conjecture for LLT polynomials.
In particular, for the LLT polynomial $\LLT_\avec(\xvec;q)$ indexed by the area sequence of
a Dyck path, we prove that the sum of the $\elementaryE$-coefficients of $\LLT_\avec(\xvec;q+1)$
equals $(q+1)^{|\avec|}$. This expression correspond to a $q$-weighted sum over orientations of a
unit interval graph.
Furthermore, we show that the analogous statements hold for vertical-strip LLT polynomials ---
this family of polynomials contain a variant of modified Hall--Littlewood polynomials.

\item We prove that the LLT polynomials associated with circular arc digraphs are symmetric,
using a proof that avoids superization. This gives a slightly simpler proof of LLT symmetry
compared to that in \cite{Haglund2005Macdonald}.

\item We prove several cases of the $\elementaryE$-positivity conjecture in the case of 
path, cycle and complete graphs, both in the chromatic and the LLT setting. 
In particular, we give combinatorial interpretations of the $\elementaryE$-coefficients. The chromatic versions have been independently considered in~\cite{Ellzey2016}, see Remark~\ref{rem:ellzey}. 
\end{itemize}

Note that when we discuss LLT polynomials in this paper, we will mainly treat the evaluation 
$\LLT_\avec(\xvec;q+1)$, where the $q$-parameter has been shifted by $1$, 
which turns out to be the natural setting to work in.
In \cref{tab:chromatic-vs-LLT}, we show the mirror correspondence between
chromatic quasisymmetric functions and unicellular LLT polynomials.
\begin{table}[h!]
\setlength{\tabcolsep}{8pt}
\centering
\begin{tabular}{lll}
\toprule
Property & Chromatic & $1$-shifted LLT \\
\midrule
Schur-positive & Yes\textsuperscript{$\ast$} &  Yes\textsuperscript{$\ast$} \\
Positive $e$-expansion & Conjectured & Conjectured\\
$e$-coefficients  & Acyclic orientations & $q$-acyclic orientations \\
Fixed length $e$-coefficients & Number of sinks & Number of half-sinks \\
$\omega(f)$ is $p$-positive & Yes\textsuperscript{\textdagger} &  Yes\textsuperscript{\textdagger\textdagger} \\
\bottomrule
\end{tabular}
\caption{The mirror correspondence.
\textsuperscript{$\ast$}Is only known in the unit interval case.
\textsuperscript{\textdagger}Recently proved in the unit interval setting for the chromatic quasisymmetric functions,
see \cite{Athanasiadis15}, and in the cyclic case in \cite{Ellzey2016}.
\textsuperscript{\textdagger}We give the $\psumP$-expansion of unicellular LLT polynomials in this paper.
}\label{tab:chromatic-vs-LLT}
\end{table}

The diagram in \cref{fig:polynomialGraph} illustrates the families of polynomials we consider in this paper.
\begin{figure}[!ht]
\includegraphics[width=0.9\textwidth]{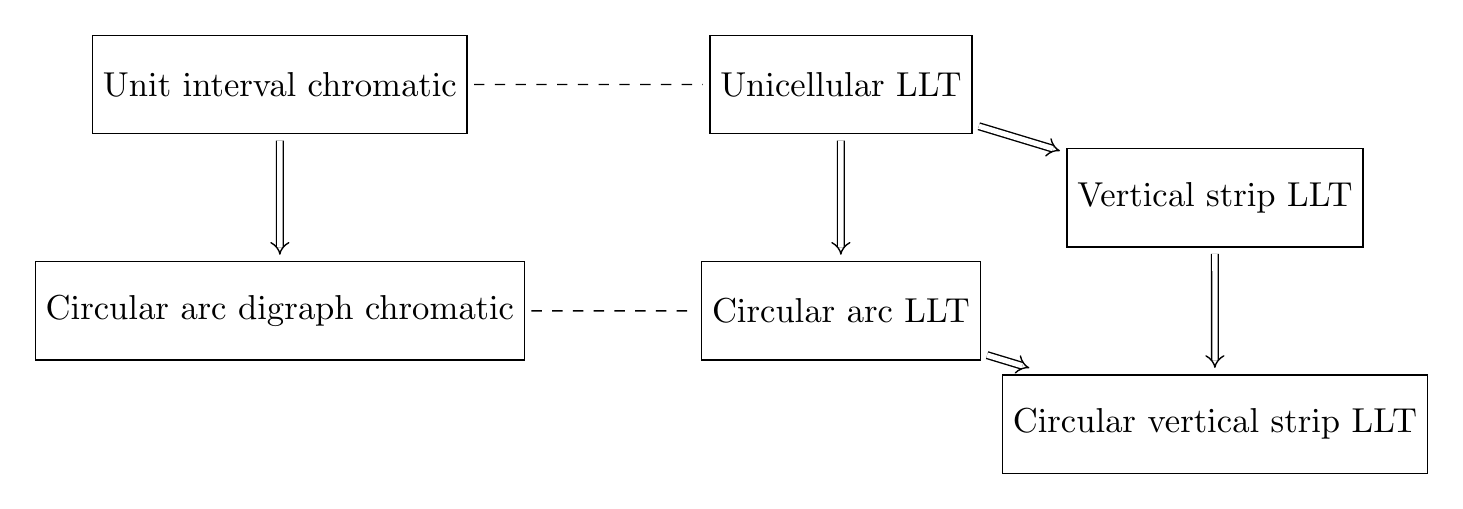}
\caption{The thick arrows represent the superset relation,
while the dashed line indicate the ``mirror correspondence''.
The left hand side consists of polynomials given as sum over proper colorings, while in the right hand side,
all colorings are allowed, with exception in the vertical strip cases where certain inequalities are enforced.
}\label{fig:polynomialGraph}
\end{figure}

Another motivation for this work is to try to unify two open problems regarding LLT polynomials 
and chromatic quasisymmetric functions: give combinatorial proofs for Schur 
positivity of LLT polynomials (which is still open) and $e$-positivity of chromatic symmetric functions.

\begin{remark}\label{rem:ellzey}
We note that the extension regarding the chromatic quasisymmetric functions to circular arc digraphs 
has also been considered independently in \cite{Ellzey2016},
and give the same combinatorial $\elementaryE$-expansion for the path and the cycle.
Furthermore, the power-sum expansion in the circular arc digraph case of
chromatic quasisymmetric functions appears in \cite{Ellzey2016}.
\end{remark}


\section{The setup: Dyck paths, posets and unit interval graphs}

We use standard notation: $[n]$ is the set $\{1,2,\dotsc,n\}$ and $[n]_q$ is the $q$-integer $1+q+ \dotsb +q^{n-1}$. 
Vectors of numbers or variables, or sequences of partitions, will be denoted in bold, \emph{e.g.} $\xvec=(x_1,\dotsc,x_n)$ 
and $\nuvec = (\nu^1,\dotsc)$ is a sequence of partitions. 
The quasisymmetric chromatic polynomial for a graph $\Gamma_\avec$ is denoted $X_\avec$ and the
LLT polynomials is denoted $\LLT_\avec$.

We begin by defining our main family of graphs, which generalizes unit-interval graphs:
\begin{definition}\label{def:cyclicUIGraph}
A \defin{circular unit arc digraph} is a \emph{directed graph} with vertex set $[n]$ and edges
\begin{align}\label{eq:cyclicUIGraph}
 i \to i+1, \; i \to i+2,\; \dotsc, i \to i+a_i
\end{align}
for all $i = 1,\dotsc,n$, where vertex indices are taken modulo $n$, and the integers $a_1,\dotsc,a_n$ satisfy
\begin{itemize}
 \item $0 \leq a_i \leq n-1 $ for $1 \leq i \leq n$,
 \item $a_{i} -1 \leq a_{i+1}$ for $1 \leq i \leq n$,
\end{itemize}
where the index is again taken mod $n$ in the second condition. 
We denote this directed graph $\Gamma_\avec$.
\end{definition}
Whenever $a_n=0$, we say that $\Gamma_\avec$ is a \emph{unit interval graph}.
The sequence $a_1,a_2,\dotsc,a_n$ is called the \emph{area sequence} of the graph,
for reasons that will be evident shortly.
A convenient way to present such unit interval graphs is by using a \emph{Dyck diagram} (in the case $a_n=0$),
or a \emph{circular Dyck diagram} for the general $\avec$. We often write \emph{circular area sequences}
to emphasize that $a_n$ is allowed to be non-zero.
\begin{example}
Consider a Dyck path as in \eqref{eq:dyckPathExample}, where the squares above the path are shaded.
The Ferrers diagram formed by these boxes are referred to as the \defin{outer shape}.
In \cref{eq:dyckPathExample}, the area sequence is given by $(2,2,3,2,1,0)$ --- the number of white boxes in each row.
The edges of $\Gamma_\avec$ are 
\[
 E(G_\avec) = \{12,13,23,24,34,35,36,45,46,56\}.
\]
It is straightforward to show that the conditions on $\avec$ together with $a_n=0$ always correspond to a Dyck path,
and that every Dyck path is obtained from some $\avec$.
In this case, we make no difference between $\avec$ and the Dyck path it represents. 
We use standard notation and let $|\avec|$ denote the sum of the entries in $\avec$ --- the total number of inner squares,
commonly known as the \emph{area} of the Dyck path.
\begin{align}\label{eq:dyckPathExample}
\begin{ytableau}
*(lightgray) \scriptstyle{16} & *(lightgray) \scriptstyle{15} & *(lightgray) \scriptstyle{14} &  \scriptstyle{13}   &  \scriptstyle{12} & *(yellow) 1 \\
*(lightgray) \scriptstyle{26} & *(lightgray) \scriptstyle{25} &   \scriptstyle{24} &  \scriptstyle{23}  & *(yellow) 2\\
  \scriptstyle{36}  &  \scriptstyle{35} &   \scriptstyle{34} & *(yellow) 3\\
  \scriptstyle{46}  &  \scriptstyle{45}  & *(yellow) 4\\
  \scriptstyle{56} & *(yellow) 5\\
*(yellow) 6
\end{ytableau}
\end{align}
Boxes that are not in the outer shape or on the diagonal are referred to as \defin{inner boxes}.
We also extend the Dyck path model to accommodate for the general area sequences in a natural manner.
For example, $\avec = (3,3,2,3,2,3)$ is illustrated as
\[
\begin{ytableau}
\none &\none &\none &\none &\none &\none & *(lightgray) & *(lightgray) & \scriptstyle{14} & \scriptstyle{13} & \scriptstyle{12} & *(yellow) 1 \\
\none &\none &\none &\none &\none & *(lightgray) & *(lightgray) & \scriptstyle{25} & \scriptstyle{24} & \scriptstyle{23} & *(yellow) 2 \\
\none &\none &\none &\none & *(lightgray) & *(lightgray) & *(lightgray)  & \scriptstyle{35} & \scriptstyle{34} & *(yellow) 3 \\
\none &\none &\none & *(lightgray) & *(lightgray) & \scriptstyle{41} & \scriptstyle{46} &\scriptstyle{45} & *(yellow) 4 \\
\none &\none & *(lightgray) & *(lightgray) & *(lightgray) & \scriptstyle{51} & \scriptstyle{56} & *(yellow) 5 \\
\none & *(lightgray) & *(lightgray) & \scriptstyle{63} & \scriptstyle{62} & \scriptstyle{61} & *(yellow) 6 \\
*(lightgray) & *(lightgray) & \scriptstyle{14} & \scriptstyle{13} & \scriptstyle{12} & *(yellow) 1 
\end{ytableau}
\]
where the bottom row is a repetition of the first row to illustrate the circular nature of the digraph.
\end{example}

Note that $\Gamma_\avec$ and $\Gamma_{\sigma(\avec)}$
are isomorphic as directed graphs when $\sigma$ is a cyclic permutation on the area sequence.
\begin{example}
There are in total $18$ circular area sequences of length $3$; the following
\[
 000, 100, 110, 111, 210, 211, 221, 222
\]
plus all cyclic permutations of these. The ones ending with $0$,
\[
 000, 100, 010, 110, 210
\]
correspond naturally to Dyck paths. 
\end{example}

Evidently, the number of unit interval graphs on $n$ vertices is given by the Catalan numbers $C_n$,
since they are in bijection with Dyck paths. We now enumerate the circular area sequences:
\begin{lemma}
The number $g(n)$ of circular area sequences of length $n$ in \cref{def:cyclicUIGraph} is given by
\[
g(n) = (n+2)\binom{2n-1}{n-1} - 2^{2n-1}.
\]
\end{lemma}
\begin{proof}
This sequence appear in OEIS \cite{OEIS} as \texttt{A194460}.
There, $g(n)$ is described as
the number of pairs, $(p,q)$ of Dyck paths of semi-length $n$
such that the first peak of $q$ has height at least $n-h_l(p)$,
where $h_l(p)$ is the height of the last peak of $p$,
and the last peak of $q$ has height at least $n-h_f(p)$,
where $h_f(p)$ is the height of the first peak of $p$.
We give a bijection between circular area sequences and such pairs of Dyck paths.
\begin{center}
\includegraphics[width=0.5\textwidth]{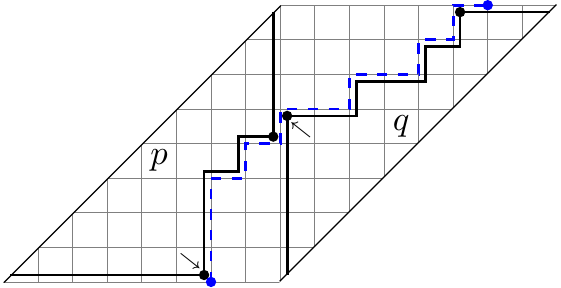}
\end{center}
In the figure above, we have drawn two Dyck paths, $p$, $q$ with the properties above.
The first and last peak of each Dyck path have been marked, with arrows pointing at the first peak
in each path. We have the conditions
\[
 h_l(p) + h_f(q) \geq n \qquad \text{ and } \qquad h_f(p) + h_l(q) \geq n
\]
which ensure that the first and the last peak on $p$ lies on the path $q$,
and vice versa. We now notice that the union of the paths between the first and last peak
on each respective path --- dashed blue in the figure --- traces out a valid circular area sequence.
Furthermore, the area sequence uniquely determines the pair $(p,q)$.
\end{proof}
We note that the numbers $g(n)$ show up in the study of types of ideals in the standard Borel
subalgebra of an untwisted affine Lie algebra, see \cite{Baur2012}.

Note also that it is clear from the Dyck diagram interpretation that
the class of circular arc digraphs is closed under taking induced subgraphs.

\begin{remark}
A direct bijective proof of the above formula for counting circular Dyck paths
has been discovered by Svante Linusson, and will appear elsewhere.
\end{remark}

\subsection{Poset interpretation}\label{ss:posetInterpretation}

For a Dyck path $\avec$, we associate a poset $P_\avec$ as follows:
let $P_\avec$ be the poset with relations $i < j$ if $ij$ is somewhere in the outer shape.
For example,
\begin{align}
\begin{ytableau}
*(lightgray) \scriptstyle{16} & *(lightgray) \scriptstyle{15} & *(lightgray) \scriptstyle{14} &  \scriptstyle{13}   &  \scriptstyle{12} & *(yellow) 1 \\
*(lightgray) \scriptstyle{26} & *(lightgray) \scriptstyle{25} &   \scriptstyle{24} &  \scriptstyle{23}  & *(yellow) 2\\
  \scriptstyle{36}  &  \scriptstyle{35} &   \scriptstyle{34} & *(yellow) 3\\
  \scriptstyle{46}  &  \scriptstyle{45}  & *(yellow) 4\\
  \scriptstyle{56} & *(yellow) 5\\
*(yellow) 6
\end{ytableau}
\qquad
\qquad
P_\avec = 
\begin{tikzpicture}[-,>=stealth',auto, baseline=(current bounding box.center),  thick,main node/.style={circle,draw,inner sep=2pt, minimum width=8pt},scale=0.5]
\node[main node] (2) at (0,2) {2};
\node[main node] (6) at (0,0) {6};
\node[main node] (1) at (2,2) {1};
\node[main node] (5) at (2,0) {5};
\node[main node] (4) at (4,0) {4};
\path[->]
(1) edge (4)
(1) edge (5)
(1) edge (6)
(2) edge (5)
(2) edge (6);
\end{tikzpicture}
\end{align}
The \defin{incomparability graph} of $P_\avec$ is then $\Gamma_\avec$.

\begin{lemma}
The poset $P_\avec$ is $(3+1)$-avoiding and $(2+2)$-avoiding and is a natural unit interval order. 
Furthermore, $\Gamma_\avec$ is an unit interval graph.
\end{lemma}
\begin{proof}
This follows from the characterization given in \cite[Prop. 4.1]{ShareshianWachs2014}. 
There is a straightforward way to go from area sequence to the $m$-sequence 
Shareshian and Wachs define (called the Hessenberg vector), namely by using the relation $a_i = m_i-i$.
In particular, we can construct an explicit one as follows by induction.
Let $I_1,\ldots,I_{n-1}$ be unit intervals, ordered increasingly by midpoints,
which respect the $\Gamma_\avec$ order for the first $n-1$ points.
Let $j_n =\min\{ j: m_j=n\}$, now let $I_n \coloneqq I_{n-1} + \epsilon$.
We have that $I_{n-1} \cap I_{j_n} = [a,b]$ for some $a<b$,
so we choose $0<\epsilon <b-a$ if $I_{n-1} \cap I_{j_n-1} = \emptyset$,
otherwise $I_n = I_{j_n-1}+1+\epsilon$ for some $\epsilon <\max I_{j_n} - \max I_{j_{n}-1}$.
\end{proof}

\subsection{Chromatic quasisymmetric functions}

A coloring $F$ of a circular unit arc digraphs $\Gamma_\avec$ is an assignment of natural numbers to the vertices.
The coloring is \emph{proper}, or \emph{non-attacking} if no two vertices connected by an edge have the same color.
An \defin{ascent} in the coloring is a directed edge $i \to j$ such that $F(i)<F(j)$.
Given an orientation $\theta$ of $\Gamma_\avec$, the edges in $\theta$ that agree with the orientation
of the corresponding edges in $\Gamma_\avec$ are called \defin{ascending edges} of $\theta$.

\medskip 

We are now ready to define chromatic quasisymmetric functions associated with circular unit arc digraphs.
\begin{definition}
The \defin{chromatic quasisymmetric function} $\chrom_\avec(\xvec;q)$ is defined as
\[
 \chrom_\avec(\xvec;q) = \sum_{\substack{F : \Gamma_\avec \to \setN \\ F \text{ non-attacking}} } x^F q^{\asc_\avec(F)}
\]
where the sum is taken over all \emph{non-attacking} colorings of $\Gamma_\avec$,
and $\asc_\avec(F)$ is the number of ascents in the coloring.
\end{definition}
In the case of a non-circular $\avec$, this definition agrees with the definition in \cite{ShareshianWachs2014}.

The chromatic quasisymmetric functions are symmetric in the case of unit interval graphs, as shown in
 \cite{ShareshianWachs2014}.
We now extend this result to the circular unit arc digraph case.
This statement is also proved in \cite{Ellzey2016}.
\begin{lemma}
The chromatic quasisymmetric function associated with a circular unit arc digraph is symmetric.
\end{lemma}
\begin{proof}
Let $\Gamma_\avec$ be a circular unit arc digraph. As in~\cite{ShareshianWachs2014}, 
consider a coloring of $\Gamma_\avec$ and the subgraph consisting of vertices colored $i$ and $i+1$.
The connected components consist of either directed chains with alternating color, or directed cycles of even length.

It is now straightforward to show that interchanging the colors $i$ and $i+1$ on all
chains of odd length preserves the $q$-weight.
Since this is possible for all $i$, $\chrom_\avec(\xvec;q)$ is symmetric.
\end{proof}

\section{LLT diagrams and LLT polynomials}\label{sec:LLTDiagrams}

LLT polynomials were introduced by Lascoux, Leclerc and Thibon in \cite{Lascoux97ribbontableaux}.
The LLT polynomials can in general be seen as a $q$-deformation of products of skew Schur functions,
and they appear as a central object in the study of modified Macdonald polynomials.
LLT polynomials are also related to generalized Kostka coefficients
and Kazhdan--Lusztig polynomials, see \emph{e.g.} \cite{LeclercThibon2000}.

A big open problem in the area is to give a combinatorial proof of Schur positivity for LLT polynomials.
A solution of this problem would immediately give a combinatorial formula for the $qt$-Kostka coefficients
that appear in the expansion of modified Macdonald polynomials in terms of Schur polynomials.

We now give the definition of LLT polynomials as it appears in \cite{Haglund2005Macdonald}:
\begin{definition}
Let $\nuvec$ be a $k$-tuple of skew Young diagrams.
Given such a tuple, we let 
$\SSYT(\nuvec) = \SSYT(\nuvec^1) \times \SSYT(\nuvec^2)\times \dotsm \times \SSYT(\nuvec^k)$
where $\SSYT(\lambda)$ is the set of skew semi-standard Young tableaux of shape $\lambda$.
Given $T = (T^1,T^2,\dotsc,T^k) \in \SSYT(\nuvec)$,
let $\xvec^T$ denote the product $\xvec_{T^1}\dotsm \xvec_{T^k}$ where $\xvec_{T^i}$
is the usual weight of the semi-standard Young tableau $T^i$.
Entries $T^i(u) > T^j(v)$ form an \emph{inversion} if
either 
\begin{itemize}
 \item $i<j$ and $c(u) = c(v)$, or
 \item $i>j$ and $c(u) = c(v)+1$, 
\end{itemize}
where $c(u)$ denotes the \emph{content} of $u$. The content of a cell $(i,j)$ in a skew diagram is $i-j$. 
Finally, we can define the LLT polynomial 
\[
\LLT_\nuvec(\xvec;q) = \sum_{T \in \SSYT(\nuvec)}  q^{\inv(T)} \xvec_T 
\]
where $\inv(T)$ is the total number of inversions appearing in $T$.
\end{definition}

\medskip 

A convenient way to visualize a tuple of skew shapes is to arrange them in the first quadrant,
such that boxes with the same content appear on the same diagonal in a non-overlapping fashion.
As an example, the \defin{LLT diagram} associated with the skew shapes
\[
 (3,2)/(1), \qquad (3,1), \qquad (3,3)/(2,1)
\]
is presented in \eqref{eq:LLTDiagramExample}, with skew boxes marked with $\times$.
Note that the diagrams are drawn in the \emph{French convention}, which is traditional for LLT polynomials.
\begin{align}\label{eq:LLTDiagramExample}
\begin{tikzpicture}[baseline=(current bounding box.center)]
\draw[step=1em, gray, very thin] (-0.001,0) grid (9em, 8em);
\draw[gray, very thin, dashed,x=1em,y=1em] (0,0) -- (8,8);
\draw[gray, very thin, dashed,x=1em,y=1em] (0,1) -- (7,8);
\draw[gray, very thin, dashed,x=1em,y=1em] (1,0) -- (9,8);
\draw[gray, very thin, dashed,x=1em,y=1em] (2,0) -- (9,7);
\node[x=1em,y=1em] (1) at (8.5, 6.5) {1};
\node[x=1em,y=1em] (2) at (5.5, 3.5) {2};
\node[x=1em,y=1em] (3) at (2.5, 0.5) {3};
\node[x=1em,y=1em] (4) at (8.5, 7.5) {4};
\node[x=1em,y=1em] (5) at (4.5, 3.5) {5};
\node[x=1em,y=1em] (6) at (1.5, 0.5) {6};
\node[x=1em,y=1em] (7) at (7.5, 7.5) {7};
\node[x=1em,y=1em] (8) at (3.5, 3.5) {8};
\node[x=1em,y=1em] (9) at (1.5, 1.5) {9};
\node[x=1em,y=1em] (10) at (3.5, 4.5) {10};
\node[x=1em,y=1em] (11) at (0.5, 1.5) {11};
\node[x=1em,y=1em] (s1) at (7.5, 6.5) {$\times$};
\node[x=1em,y=1em] (s2) at (6.5, 6.5) {$\times$};
\node[x=1em,y=1em] (s3) at (6.5, 7.5) {$\times$};
\node[x=1em,y=1em] (s4) at (0.5, 0.5) {$\times$};
\end{tikzpicture}
\end{align}
Notice the \emph{reading order} of the boxes indicated in \eqref{eq:LLTDiagramExample}.
Boxes are read in decreasing order of content and then in decreasing $x$-coordinate if 
contents are equal.
\medskip

There are a few important subfamilies of LLT polynomials, indexed by certain shapes:
\emph{ribbons}, \emph{vertical strips} and \emph{unicellular} diagrams,
related as
\[
 \text{unicellular} \subset \text{vertical strips} \subset \text{ribbons}.
\]
Ribbon skew shapes are skew diagrams without any $2\times 2$-subdiagram of boxes.
A vertical strip is a skew shape consisting of a single vertical strip of boxes,
and unicellular LLT diagrams are diagrams where each skew shape consists of a single box.
As an example, all shapes that appear in \eqref{eq:LLTDiagramExample} are ribbons.
\medskip

A ribbon LLT diagram can also be represented using a marked variant of Dyck diagrams,
by placing the boxes in the reading order along the main diagonal.
Pairs of boxes that could potentially contribute to $\inv$ are white squares in the Dyck diagram.
We note that the white squares indeed describe a region under a Dyck path, \emph{i.e.}
there are no ``holes'' in the diagram. This follows because if there are boxes $i<k$
which could contribute to an inversion then all boxes $j \in [i,k]$ can also contribute
one with $i$ and $k$: either $j$ is below $i$ on the same diagonal and hence in
inversion with $i$ and $k$, or $j$ is above $k$ on its diagonal above and again forms an inversion with both.

\begin{example}
The LLT diagram in \eqref{eq:LLTDiagramExample} is represented as 
\begin{align}\label{eq:LLTDiagramRibbonBijection}
\begin{tikzpicture}[baseline=(current bounding box.center)]
\draw[step=1em, gray, very thin] (-0.001,0) grid (9em, 8em);
\draw[gray, very thin, dashed,x=1em,y=1em] (0,0) -- (8,8);
\draw[gray, very thin, dashed,x=1em,y=1em] (0,1) -- (7,8);
\draw[gray, very thin, dashed,x=1em,y=1em] (1,0) -- (9,8);
\draw[gray, very thin, dashed,x=1em,y=1em] (2,0) -- (9,7);
\node[x=1em,y=1em] (1) at (8.5, 6.5) {1};
\node[x=1em,y=1em] (2) at (5.5, 3.5) {2};
\node[x=1em,y=1em] (3) at (2.5, 0.5) {3};
\node[x=1em,y=1em] (4) at (8.5, 7.5) {4};
\node[x=1em,y=1em] (5) at (4.5, 3.5) {5};
\node[x=1em,y=1em] (6) at (1.5, 0.5) {6};
\node[x=1em,y=1em] (7) at (7.5, 7.5) {7};
\node[x=1em,y=1em] (8) at (3.5, 3.5) {8};
\node[x=1em,y=1em] (9) at (1.5, 1.5) {9};
\node[x=1em,y=1em] (10) at (3.5, 4.5) {10};
\node[x=1em,y=1em] (11) at (0.5, 1.5) {11};
\node[x=1em,y=1em] (s1) at (7.5, 6.5) {$ $};
\node[x=1em,y=1em] (s2) at (6.5, 6.5) {$ $};
\node[x=1em,y=1em] (s3) at (6.5, 7.5) {$ $};
\node[x=1em,y=1em] (s4) at (0.5, 0.5) {$ $};
\end{tikzpicture}
\qquad\qquad
 \begin{ytableau}
*(lightgray) &*(lightgray) & *(lightgray) & *(lightgray) &  *(lightgray) &  *(lightgray) &   *(lightgray)  & *(lightgray) \wedge &    &    & *(yellow) 1 \\
*(lightgray) & *(lightgray) & *(lightgray) &  *(lightgray) &  *(lightgray) &   *(lightgray) & *(lightgray) \leq &    &    & *(yellow) 2 \\
*(lightgray) & *(lightgray) & *(lightgray) &  *(lightgray) &  *(lightgray) & *(lightgray)  \leq &    &    & *(yellow) 3 \\
*(lightgray) & *(lightgray) & *(lightgray) &  *(lightgray) & *(lightgray)  \leq &     &    & *(yellow) 4 \\
*(lightgray) & *(lightgray) & *(lightgray) & *(lightgray) \leq &    &    & *(yellow) 5 \\
*(lightgray) & *(lightgray) & *(lightgray) \wedge &    &    & *(yellow) 6 \\
*(lightgray)  &  *(lightgray) &    &    & *(yellow) 7 \\
*(lightgray)& *(lightgray)\wedge  &    & *(yellow) 8 \\
*(lightgray) \leq  &    & *(yellow) 9 \\
   & *(yellow) 10 \\
 *(yellow) 11
\end{ytableau}
\end{align}
where edges marked with $\wedge$ and $\leq$ indicate the strict and weak inequalities 
that must hold between corresponding entries in order for the filling to consist 
of semi-standard Young tableaux.  
\end{example}

Due to the nature of the reading order, we note that inversions in $T \in \SSYT(\nu)$
are mapped to ascending edges in the corresponding coloring of the Dyck diagram.
We emphasize that ascending edges marked with $\wedge$ are \emph{not} counted as ascents in the coloring.
\begin{example}
Here is an example of this correspondence with vertical strips, and another with unicellular LLT diagram:
\begin{align}
\begin{tikzpicture}[baseline=(current bounding box.center)]
\draw[step=1em, gray, very thin] (-0.001,0) grid (12em, 12em);
\draw[gray, very thin, dashed,x=1em,y=1em] (0,0) -- (12,12);
\draw[gray, very thin, dashed,x=1em,y=1em] (0,1) -- (11,12);
\draw[gray, very thin, dashed,x=1em,y=1em] (0,2) -- (10,12);
\node[x=1em,y=1em] (1) at (8.5, 8.5) {1};
\node[x=1em,y=1em] (2) at (2.5, 2.5) {2};
\node[x=1em,y=1em] (3) at (0.5, 0.5) {3};
\node[x=1em,y=1em] (4) at (10.5, 11.5) {4};
\node[x=1em,y=1em] (5) at (8.5, 9.5) {5};
\node[x=1em,y=1em] (6) at (4.5, 5.5) {6};
\node[x=1em,y=1em] (7) at (2.5, 3.5) {7};
\node[x=1em,y=1em] (8) at (8.5, 10.5) {8};
\node[x=1em,y=1em] (9) at (6.5, 8.5) {9};
\node[x=1em,y=1em] (10) at (4.5, 6.5) {10};
\node[x=1em,y=1em] (11) at (2.5, 4.5) {11};
\end{tikzpicture}
&\qquad
\begin{ytableau}
*(lightgray) &*(lightgray) & *(lightgray) & *(lightgray) &  *(lightgray) &  *(lightgray) & *(lightgray)  \wedge  &  &   &  & *(yellow) 1 \\
*(lightgray) & *(lightgray) & *(lightgray) &  *(lightgray) & *(lightgray) \wedge &     &   &    &    & *(yellow) 2 \\
*(lightgray) & *(lightgray) & *(lightgray) &  *(lightgray) &    &     &    &    & *(yellow) 3 \\
*(lightgray) & *(lightgray) & *(lightgray) &  *(lightgray) &     &     &    & *(yellow) 4 \\
*(lightgray) & *(lightgray) & *(lightgray) & *(lightgray) \wedge &    &    & *(yellow) 5 \\
*(lightgray) &*(lightgray) \wedge &    &    &    & *(yellow) 6 \\
*(lightgray)\wedge  &    &    &    & *(yellow) 7 \\
   &    &    & *(yellow) 8 \\
   &    & *(yellow) 9 \\
   & *(yellow) 10 \\
 *(yellow) 11
\end{ytableau} \label{eq:verticalStripLLT}
\end{align}
\begin{align}
\begin{tikzpicture}[baseline=(current bounding box.center)]
\draw[step=1em, gray, very thin] (-0.001,0) grid (11em, 10em);
\draw[gray, very thin, dashed,x=1em,y=1em] (0,0) -- (10,10);
\draw[gray, very thin, dashed,x=1em,y=1em] (1,0) -- (11,10);
\draw[gray, very thin, dashed,x=1em,y=1em] (2,0) -- (11,9);
\node[x=1em,y=1em] (7) at (0.5, 0.5) {7};
\node[x=1em,y=1em] (5) at (2.5, 1.5) {5};
\node[x=1em,y=1em] (2) at (4.5, 2.5) {2};
\node[x=1em,y=1em] (6) at (5.5, 5.5) {6};
\node[x=1em,y=1em] (4) at (7.5, 6.5) {4};
\node[x=1em,y=1em] (1) at (9.5, 7.5) {1};
\node[x=1em,y=1em] (3) at (10.5,9.5) {3};
\end{tikzpicture}
&\qquad
\begin{ytableau}
*(lightgray) &*(lightgray) & *(lightgray) & *(lightgray)  &    &    & *(yellow) 1\\
*(lightgray) &*(lightgray) & *(lightgray) &   &   & *(yellow) 2\\
*(lightgray) &*(lightgray)  &    &    & *(yellow) 3\\
*(lightgray) &*(lightgray)  &     & *(yellow) 4\\
*(lightgray) &              & *(yellow) 5\\
   &*(yellow) 6 \\
*(yellow) 7
\end{ytableau}\label{eq:lltbij}
\end{align}
\end{example}
In \cref{prop:LLTDiagrambijection} below, we show that this is indeed
a bijection --- vertical strip LLT polynomials of degree $n$
are in bijection with Schröder paths and unicellular LLT polynomials are
enumerated by the Catalan numbers.
We need some terminology in order to carry out the bijection.

A \emph{corner edge} of a (circular) unit interval digraph $\Gamma$ is
an edge $i \to j$ which is not an edge of $\Gamma$, but $i \to j-1$ and $i+1 \to j$
are both edges of $\Gamma$. As usual, vertex indices are taken mod $n$ if necessary.

Given a Dyck diagram as in \eqref{eq:lltbij2},
read the labeled vertices in increasing order,
and greedily partition them into complete subgraphs.
In our example, vertices $1$,$2$ and $3$ form a complete subgraph, but not $\{1,2,3,4\}$.
The next two vertices, $4$, and $5$ form a complete subgraph and finally, $6$ and $7$ form the third complete subgraph.
We have marked the edges in the complete subgraphs with bullets --- 
the Dyck path immediately above the edges with bullets is commonly referred to as the \emph{bounce path}.
This definition is also extended to circular Dyck diagrams, as shown in the second diagram in \eqref{eq:lltbij2}.
\begin{align}\label{eq:lltbij2}
\begin{ytableau}
*(lightgray) &*(lightgray) & *(lightgray) & *(lightgray)  &  \bullet  & \bullet   & *(yellow) 1\\
*(lightgray) &*(lightgray) & *(lightgray) &   & \bullet  & *(yellow) 2\\
*(lightgray) &*(lightgray)  &    &    & *(yellow) 3\\
*(lightgray) &*(lightgray)  &  \bullet   & *(yellow) 4\\
*(lightgray) &              & *(yellow) 5\\
 \bullet  &*(yellow) 6 \\
*(yellow) 7
\end{ytableau}
\qquad
\begin{ytableau}
\none &\none &\none &\none &\none &\none &\none & *(lightgray) & *(lightgray) & \bullet & \bullet & \bullet  & *(yellow) 1 \\
\none &\none &\none &\none &\none &\none & *(lightgray) & *(lightgray) & *(lightgray) & \bullet &\bullet & *(yellow) 2 \\
\none &\none &\none &\none &\none & *(lightgray) & *(lightgray) & *(lightgray)  &   & \bullet  & *(yellow) 3 \\
\none &\none &\none &\none & *(lightgray) & *(lightgray) & *(lightgray) & & & *(yellow) 4 \\
\none &\none &\none & *(lightgray) & *(lightgray) & *(lightgray) & *(lightgray) & \bullet & *(yellow) 5 \\
\none &\none & *(lightgray) & *(lightgray) & & &  & *(yellow) 6 \\
\none & *(lightgray) & *(lightgray) & & & & *(yellow) 7
\end{ytableau}
\end{align}

\begin{proposition}\label{prop:LLTDiagrambijection}
Any ribbon LLT diagram can be put in correspondence with a Dyck diagram,
and then marking some of the corner edges as strict or weak.
\end{proposition}
\begin{proof}
Let $\nuvec$ denote an LLT diagram, consisting of ribbons with $n$ boxes in total.
To construct the corresponding Dyck diagram, put the boxes of the LLT diagram in their reading order along the diagonal.
By the more general argument above, the potential LLT diagram inversions correspond to the region
under a Dyck path --- if $i<k$ could contribute to an inversion and $i<j<k$, then $i<j$ and $j<k$ are also 
in such attacking order and can contribute  to an inversion.

The ribbon SSYTs require strict inequalities between boxes $i<k$, appearing on top of each other in the Young diagram.
Suppose $j$ is another box in the LLT diagram appearing between $i$ and $k$ in reading order.
It is straightforward to see that both $(i,j)$ and $(j,k)$ are potential inversions in the LLT diagram,
and it follows that corresponding edges in the Dyck diagram are under the Dyck path.
It follows that $i \to k$ is a corner edge, which we then mark as strict, to enforce the inequality between
vertices $i$ and $k$ in the Dyck diagram.

Similarly, we need to enforce weak inequalities between boxes $i \geq k$, with $i$ appearing to the right of $k$
in a ribbon (and $i$ before $k$ in reading order). A similar analysis to previous case shows
that $i \to k$ is a corner edge in the Dyck diagram, which is then marked as weak.

In the opposite direction, given such a Dyck diagram with some weak and strict corner edges,
we can construct the ribbon strip LLT as follows: 
\begin{itemize}
 \item The box $u$ is placed on LLT diagonal with content $-c$ if corresponding vertex 
  $u$ is in the $c$th complete subgraph of the bounce path.
 \item Boxes placed on the same LLT diagonal are ordered (in reading order) according to vertex label.
 \item Entries in adjacent diagonals are riffled according to the Dyck diagram.
 This means that a box $u$ is placed below $v$ if $u$ and $v$ are on adjacent diagonals, and $u$ and $v$ form a potential inversion
 in the Dyck diagram.
\end{itemize}
It is clear from the properties of Dyck diagrams that these three conditions can always be fulfilled.
For example, the first property ensures that the edges determined by the bounce path 
are present in the LLT diagram as potential inversions among boxes with same content.

Finally, boxes in adjacent diagrams are ``nudged'' immediately adjacent, or on top of each other,
according to the weak and strict corner edges, by sliding them along the diagonal.
\end{proof}
As an example, the Dyck diagram in \eqref{eq:LLTDiagramRibbonBijection} is mapped to the corresponding LLT diagram.
Note that we only prove that every ribbon LLT diagram can be represented as a marked Dyck diagram
and vice versa --- the maps are not inverses of each other, since different LLT diagrams (with the same LLT polynomial)
might be mapped to the same Dyck diagram.

\begin{corollary}
The number of (non-circular) vertical-strip LLT polynomials on $n$ vertices is given by
the small Schröder numbers, \texttt{A001003}: $1, 3, 11, 45, 197,\dotsc$
\end{corollary}

By allowing circular unit interval digraphs and marking some corner edges,
we obtain a circular extension of vertical-strip LLT diagrams.
\begin{openproblem}
The number of circular vertical-strip diagrams of size $n$ are given by
$1, 9, 65, 449, 3009, 19721,\dotsc$.
Find a closed formula, or a generating function for these numbers.
\end{openproblem}

The above bijection allow us to give an alternative definition of unicellular LLT polynomials,
as well as ribbon LLT polynomials.
Furthermore, we allow the definitions to extend to the circular setting,
thus extending the family of LLT polynomials:
\begin{definition}
The (circular) unicellular LLT polynomial $\LLT_{\avec}(\xvec;q)$ is defined as
\[
 \LLT_{\avec}(\xvec;q) = \sum_{F : \Gamma_\avec \to \setN } x^F q^{\asc_\avec F}
\]
where the sum is over \emph{all} colorings of $\Gamma_\avec$.
\end{definition}
The classical unicellular LLT polynomials correspond to the case when $\Gamma_\avec$ is a non-circular unit interval graphs.
These polynomials recently appeared in the paper \cite[Section 3]{CarlssonMellit2015},
and were defined in the same manner as here. There they referred to these polynomials
as characteristic functions in the Dyck path algebra.

A proof that $\LLT_{\avec}(\xvec;q)$ are symmetric functions for unit interval $\Gamma_\avec$
can be found in \cite{Haglund2005Macdonald}, but one needs to translate this definition to
the classical definition of LLT polynomials. We give a modified proof of symmetry 
below in \cref{subsec:LLTsymmetry} that extends to the circular setting.

\begin{remark}
Note that in the case when $\Gamma_\avec$ contains a directed cycle,
$\LLT_{\avec}(\xvec;q)$ does not belong to the classical family of LLT polynomials.
In fact, $\LLT_{\avec}(\xvec;q)$ is not always Schur positive or even positive in the 
fundamental quasisymmetric basis in the circular arc digraph setting.
\end{remark}

We also extend the definition of ribbon LLT polynomials, where the underlying graph may contain cycles.
\begin{definition}
Let $\nuvec \coloneqq (\avec,\svec,\wvec)$ define a circular Dyck diagram $\avec$,
where some corner edges $\svec$ are marked as strict, and some other corner edges $\wvec$ marked weak.
The \emph{circular ribbon LLT polynomial} $\LLT_{\nuvec}(\xvec;q)$ is defined as
\[
 \LLT_{\nuvec}(\xvec;q) \coloneqq \sum_{F : \Gamma_\nuvec \to \setN } \xvec^F q^{\asc_\nuvec F}
\]
where the sum is over all colorings of $\Gamma_\nuvec$, which are strict on $\svec$ and weak on $\wvec$.
That is, $(u\to v) \in \svec$ implies $F(u)<F(v)$ and $(u\to v) \in \wvec$ implies $F(u)\geq F(v)$.
\end{definition}
As before, this definition coincides with the previous definition of ribbon LLT polynomials in the non-circular setting.

\subsection{Proof of symmetry for LLT polynomials}\label{subsec:LLTsymmetry}

It is a bit more of a challenge to show symmetry of the ribbon LLT polynomials in the circular case.
The following proof uses the same techniques as in \cite[Lem. 10.2]{Haglund2005Macdonald},
however, we avoid the need for ``superization'' with a second set of variables.

\begin{proposition}
Every circular ribbon polynomial $\LLT_\nuvec(\xvec;q)$ is symmetric.
\end{proposition}
\begin{proof}
It suffices to prove that $\LLT_\nuvec(\xvec;q)$ is symmetric in $x_i$ and $x_{i+1}$,
for all $i$. Given a coloring $F$, let $T$ be the entries with color $\{i,i+1\}$ and $F \setminus T$
be the remaining entries. We have that
\begin{align}\label{eq:reduceToTwoColors}
 \xvec^F q^{\asc_{\nuvec} F } = \xvec^{(F \setminus T)} q^{\asc_{\nuvec}(F,T)}
 q^{\asc_{\nuvec}(T)} \xvec^T
\end{align}
where $\asc_{\nuvec}(F,T)$ denote ascents involving at most one of the colors $i$ and $i+1$,
and $\asc_{\nuvec}(T)$ is the number of ascents where both colors are in $\{i,i+1\}$.
Note that $\asc_{\nuvec}(F,T)$ only depend on $F \setminus T$.
It follows that it suffices to prove symmetry for colorings involving only two colors, $1$ and $2$.
\medskip 

Note that a forced weak inequality can be reproduced by a difference of polynomials
involving a strict inequality:
\[
\begin{ytableau}
*(lightgray) \leq & \none[\cdots] & u \\
\none[\vdots] \\
v \\
\end{ytableau}
\qquad
=
\qquad
\begin{ytableau}
*(lightgray) & \none[\cdots] & u \\
\none[\vdots] \\
v \\
\end{ytableau}
\qquad
-
\qquad
\begin{ytableau}
*(lightgray) \wedge & \none[\cdots] & u \\
\none[\vdots] \\
v \\
\end{ytableau}
\]
By repeating this reduction recursively, it suffices to prove symmetry 
for vertical-strip LLT polynomials in two variables.
Every strict edge fixes the colors of the endpoints, and since the colors are opposite, an $x_1x_2$ can be factored out
by using an argument similar to the one in \eqref{eq:reduceToTwoColors}.
Therefore, it remains to show symmetry for the (circular) unicellular LLT polynomials.

Consider a circular Dyck diagram $\avec$ on $n$ vertices.
We do induction over $n$ and $|\avec|$.
The cases $n=0$ or $1$ are trivial and it is straightforward
to see that if $|\avec|=0$, $\LLT_\avec(x_1,x_2;q)$ is simply $(x_1+x_2)^n$.

Suppose now that $|\avec|>0$ which means that there is an inner square somewhere.
We can pick this inner square $\young(\ast)$ such that the box above it and the box to the left are not inner.
This condition implies that if we let $\young(\ast)$ to part of the outer shape,
the resulting shape $\bvec$ defines a valid circular unit arc digraph $\Gamma_\bvec$.
\[
 \begin{ytableau}
*(lightgray) &*(lightgray)   & \none[\cdots] \\
*(lightgray) & \ast & \none[\cdots] & \none[\cdots] & u\\
\none[\vdots]& \none[\vdots] \\
\none[\vdots]& \none[\vdots] \\
\none        &    v \\
\end{ytableau}
\]
In other words, $\young(\ast)$ is a corner and it corresponds to an edge $u \to v$ in $\Gamma_\avec$.
By cyclic relabeling of the graph, we can assume that $u < v$ as vertex labels.

Consider a coloring $F$ of $\Gamma_\avec$ and $\Gamma_\bvec$.
It is quite clear that
\begin{align}
 \asc_\avec(F) =
 \begin{cases}
 \asc_\bvec(F)+1 & \text{ if } F(u)=1, F(v)=2 \\
 \asc_\bvec(F) & \text{ otherwise}.
 \end{cases}
\end{align}
A coloring in the first case has the property that every vertex between
$v$ and $u$ form an ascend with either $u$ or $v$, \emph{independent of the coloring} $F$.
Furthermore, $u$ and $v$ cannot form any other ascends with vertices outside this interval.

Let $\cvec$ denote the circular unit arc digraph obtained from $\avec$ where $u$ and $v$
have been removed. We now have that
\begin{align}\label{eq:llt2variableRecursion}
 \LLT_\avec(x_1,x_2;q) =  \LLT_\bvec(x_1,x_2;q) + q^{v-u-1}(q-1)x_1x_2 \LLT_\cvec(x_1,x_2;q),
\end{align}
since every coloring of $\Gamma_\avec$ can be created from a coloring of $\Gamma_\bvec$,
but we need to modify the $q$-weight of the colorings where $u \to v$ is an ascent in $\Gamma_\avec$.
Such colorings are obtained from a coloring of $\Gamma_\cvec$,
inserting vertices $u$ and $v$ with $1$ and $2$ and compensating for the extra ascends, $v-u-1$ of them.
The factor $(q-1)$ corresponds to choosing if $u \to v$ is included as an edge or not.

By induction hypothesis, $\LLT_\bvec(x_1,x_2;q)$ is symmetric since $|\bvec|+1 = |\avec|$,
and $ \LLT_\cvec(x_1,x_2;q)$ is symmetric since $\Gamma_\cvec$ has fewer vertices than $\Gamma_\avec$.
\end{proof}

In the following special case, one can produce a  simple involution that shows the statement.
\begin{lemma}\label{lem:completeGraphCase}
Suppose $\avec$ is the unit interval graph with $\avec = (n-1,n-2,\dotsc,0)$.
Then $\LLT_\avec(x_1,x_2;q)$ is symmetric.
\end{lemma}
\begin{proof}
Consider the subword on the diagonal in the Dyck path filling consisting of the letters $i$ and $i+1$.
Reverse this subword, and replace every instance of $i$ with $i+1$ and vice versa.
It is easy to see that this map preserves the number of ascends.
\end{proof}

\section{Some properties of LLT polynomials}

We now phrase some properties of LLT polynomials in the Dyck path model, and 
relate the LLT polynomials to the multivariate Tutte polynomial of Stanley.
Suppose $\Gamma_\avec$ is a unit interval graph. The \emph{transpose} of $\avec$, denoted $\avec^T$,
is the transposed diagram of $\avec$, as illustrated in \eqref{eq:diagramTransposition}. 
Furthermore, we define the transpose of an edge $(i,j)$, to be the edge $(n+1-j,n+1-i)$. 

The following is a consequence of \cite[Lemma 10.1]{Haglund2005Macdonald}:
\begin{lemma}
Let $(\avec,\svec,\wvec)$ denote a (non-circular) unit interval graph with some corners marked strict or weak.
Then
\[
\omega \LLT_{(\avec,\svec,\wvec)}(\xvec;q) = q^{|\avec|}\LLT_{(\avec^T,\wvec^T,\svec^T)}(\xvec;q^{-1}).
\]
Note that the role of weak and strict edges have been interchanged.
\end{lemma}
\begin{example}
The following illustrates the action of $\omega$ on the area sequence
and the marked edges: $(\avec,\svec,\wvec)$ is sent to $(\avec^T,\wvec^T,\svec^T)$.
\begin{align}\label{eq:diagramTransposition}
\begin{ytableau}
*(lightgray) & *(lightgray) & *(lightgray) \wedge &    &    & *(yellow) 1\\
*(lightgray) & *(lightgray) \leq &   &   & *(yellow) 2\\
*(lightgray)  &    &    & *(yellow) 3\\
*(lightgray) \leq &     & *(yellow) 4\\
              & *(yellow) 5\\
*(yellow) 6 \\
\end{ytableau}
\qquad
\stackrel{\omega}{\longrightarrow}
\qquad
\begin{ytableau}
*(lightgray) & *(lightgray) & *(lightgray)  &  *(lightgray) \wedge &    & *(yellow) 1\\
*(lightgray) & *(lightgray) \wedge &   &   & *(yellow) 2\\
*(lightgray) \leq &    &    & *(yellow) 3\\
 &     & *(yellow) 4\\
 & *(yellow) 5\\
*(yellow) 6 \\
\end{ytableau}
\end{align}
\end{example}

\begin{question}
Can this be generalized to the circular arc setting?
\end{question}

\begin{remark}
We should mention that the top degree component (in $t$) of the modified Macdonald $\macdonaldH_\lambda(\xvec;q,t)$
is given by a vertical-strip LLT polynomial, 
and the degree $0$ term is a modified Hall--Littlewood polynomial. 
The latter can be given as certain horizontal-strip LLT polynomials,
see \cite{Haglund2005Macdonald,qtCatalanBook} for details.
\end{remark}

Recall the definition of the multivariate Tutte polynomial, \cite{Stanley98Chromatic},
defined as 
\[
 \mathrm{Tutte}_\avec(\xvec;q) = \sum_{F:\Gamma_\avec \to \setN} \xvec^F (1+q)^{m(F)}
\]
where $m(F)$ is the number of monochromatic edges in the coloring of $\Gamma_\avec$.
These polynomials have nice properties (positive $\psumP$-expansion),
and it is therefore natural to consider the LLT polynomials with $q$ shifted by $1$: 
\[
 \LLT_{\nuvec}(\xvec;q+1) = \sum_{F : \Gamma_\nuvec \to \setN } \xvec^F (1+q)^{\asc_\nuvec F}.
\]

The main conjecture in this paper is the following:  
\begin{conjecture}\label{conj:LLTvertStripEpos}
Let $\nuvec = (\avec,\svec)$ be a circular Dyck diagram with some strict corner edges.
Then $\LLT_\nuvec(\xvec;q+1)$ is $\elementaryE$-positive with unimodal coefficients.
\end{conjecture}
Below, we provide several results that supports this conjecture, for example
\cref{prop:vert-strip-llt-e-positivity}.

\section{Expansions in the elementary symmetric functions}

The main open problem regarding chromatic quasisymmetric functions is the following conjecture, stated in~\cite{ShareshianWachs2014}:
\begin{conjecture}\label{conj:chromatic-e-positivity}
Let $\Gamma_\avec$ be a unit interval graph.
Then $\chrom_\avec(\xvec;q)$ is $\elementaryE$-positive, with unimodal coefficients.
\end{conjecture}

We strongly suspect the same statement generalizes to circular unit arc digraphs.
This has also been conjectured in \cite{Ellzey2016}.
\begin{conjecture}
Let $\avec$ be a circular area sequence.
Then $\chrom_\avec(\xvec;q)$ is $\elementaryE$-positive, with unimodal coefficients.
\end{conjecture}

\medskip 

There are some promising steps towards proving this conjecture.
A $q$-adaptation of a result in \cite{Stanley95Chromatic},
appears in \cite{ShareshianWachs2011}, which deals with unit interval graphs.
The same proof strategy goes through without modification, also noted in \cite{Ellzey2016}:
\begin{proposition}\label{prop:chromaticAcyclicOrientations}
Let $\avec$ be a circular area sequence and consider the expansion
\begin{align}\label{eq:chromeexpansion}
 \chrom_\avec(\xvec;q) = \sum_\mu c^\avec_\mu(q) \elementaryE_\mu(\xvec).
\end{align}
The coefficients $c^\avec_\mu(q)$ satisfy 
\begin{align}\label{eq:chromeexpansion-with-sinks}
 \sum_{\mu} c^\avec_\mu(q) t^{\length(\mu)} = \sum_{ \theta : AO(\Gamma_\avec) } q^{\asc_\avec(\theta)} t^{\sinks(\theta)}.
\end{align}
Here, $AO(\Gamma_\avec)$ is the set of acyclic orientations of $\Gamma_\avec$,
$\asc_\avec(\theta)$ is the number of ascending edges of $A$ and $\sinks(\theta)$ 
denotes the number of sinks in the acyclic orientation.
\end{proposition}
%

What now follows is an LLT-analogue of \cref{prop:chromaticAcyclicOrientations},
but we need to some terminology first in order to state the proposition.
Let $\Gamma_\nuvec$ be a (circular) vertical strip graph (\emph{i.e.} the circular graph corresponding to a
collection of vertical strips $\nuvec$ in the LLT representation), and let $\theta$ be an orientation of $\Gamma_\nuvec$.
A \defin{half-sink} of $\theta$ is a vertex $v$, such that for all edges $v \to u$ in $\Gamma_\nuvec$,
we have $u \to v$ in $\theta$. Let $\halfsinks(\theta)$ denote the number of such half-sinks.
In other words, in the diagram representation of $\theta$, if $v$ is a half-sink in $\theta$,
then all boxes to the left of $v$ are pointing towards $v$.
Similarly, a \defin{half-source} of $\theta$ is a vertex such that for all edges $v \to u$ in $\Gamma_\nuvec$,
we have $v \to u$ in $\theta$.
For example, in the following orientation, vertices $3$, $5$ and $6$ are half-sinks,
and $1$ and $6$ are half-sources.
\begin{align*}
\begin{ytableau}
*(lightgray) &  *(lightgray) &  *(lightgray) & \downarrow  &  \downarrow & *(yellow) 1 \\
*(lightgray) &  *(lightgray) \downarrow & \downarrow  & \rightarrow  & *(yellow) 2\\
*(lightgray) &  \rightarrow & \rightarrow  & *(yellow) 3\\
 \rightarrow  &  \downarrow & *(yellow) 4\\
\rightarrow  & *(yellow) 5\\
*(yellow) 6
\end{ytableau}
\end{align*}

We are now ready to state the LLT analogue of \cref{prop:chromaticAcyclicOrientations}.
The proof closely follows the one in \cite{Stanley95Chromatic}, but we need to do some modifications:
\begin{proposition}\label{prop:vert-strip-llt-e-positivity}
Let $\nuvec$ be a circular unit interval graph with some strict corner edges.
Consider the expansion of the circular vertical-strip LLT polynomial
$\LLT_\nuvec(\xvec;q+1) = \sum_\mu d^\nuvec_\mu(q) \elementaryE_\mu(\xvec)$.
Then
\begin{align}\label{eq:llt-e-coefficient-sum}
 \sum_{\mu} d^\nuvec_\mu(q) t^{ \length(\mu) } 
 = \sum_{\theta:O_\ast(\Gamma_\nuvec)} q^{\inv(\theta)} t^{\halfsources(\theta)}
 = \sum_{\theta:O_\ast(\Gamma_\nuvec)} q^{\inv(\theta)} t^{\halfsinks(\theta)}
\end{align}
where $O_\ast(\Gamma_\nuvec)$ is the set of orientations of $\Gamma_\nuvec$ such 
that the subgraph consisting of the ascending edges is acyclic
when all strict corner edges are oriented in an ascending fashion.
\end{proposition}
\begin{proof}
We first derive an alternative expression for $\LLT_\nuvec(\xvec;q+1)$:
\begin{align}\label{eq:llt-as-orientation-sum}
\sum_{ F : \Gamma_\nuvec \to \setN } (1+q)^{\asc F} \xvec^F = \sum_{ \substack{\theta:O_\ast(\Gamma_\nuvec)} } q^{\asc_\nuvec(\theta)} X_{\theta}
\end{align}
where
\begin{align}\label{eq:XthetaDef}
X_{\theta} = \sum_{ \substack{ F : \Gamma_\nuvec \to \setN \\  F \text{ is $\theta$-compatible}}} \xvec^F.
\end{align}
Let $\theta$ be an orientation of $\Gamma_\nuvec$. A coloring $F$ is \emph{$\theta$-compatible} if for
every ascending edge $i\to j$ in $\theta$,
we have $F(i) < F(j)$. The number of ascents of the coloring depends only on $\theta$
and is given by $\asc_\nuvec(\theta)$.
A fixed coloring $F$ might contribute to several $X_\theta$, and it is clear that it is
impossible to have a coloring that is compatible with a cycle of ascending chain.
Hence, colorings can only be compatible with orientations in $O_\ast(\Gamma_\nuvec)$.

The left-hand side of \cref{eq:llt-as-orientation-sum} correspond to choosing a coloring,
then choosing a subset of the ascending edges of the coloring that contribute to the $q$-weight.
The right hand side corresponds to first choosing the contributing edges (the orientation $\theta$) and then
summing over all colorings compatible with this choice. This establish the identity \cref{eq:llt-as-orientation-sum}.

Note that $X_\theta$ is a quasi-symmetric function.
In fact, consider only the ascending edges in $\theta$.
These define an acyclic orientation on $\Gamma_\nuvec$,
and therefore, the transitive closure of these ascending edges gives a poset $P(\theta)$ on $[n]$.

\medskip
We now  follow R.\ Stanley, \cite{Stanley95Chromatic}. Let $P$ be a poset on $[n]$ and let
\begin{align}\label{eq:posetQSfunc}
X_P = \sum_{F : [n] \to \setN } x_{F(1)} \dotsm x_{F(n)}
\end{align}
summed over all strict order-preserving\footnote{Stanley does order-reversing maps. We have modified the statements accordingly.} 
maps $F: P \to \setN$, \emph{i.e.}, $i <_P j$ implies $F(i) < F(j)$.
Comparing the definitions, we see that $X_\theta = X_{P(\theta)}$.
Define the following linear transform on quasi-symmetric functions,
here defined on the basis of the \emph{fundamental} quasi-symmetric functions:
\begin{align}
 \phi(Q_S(\xvec)) =
 \begin{cases}
 t(t-1)^i &\text{ if } S = {i+1,i+2,\dotsc,n}, \\
 0 &\text{ otherwise}.
 \end{cases}
\end{align}
In~\cite{Stanley95Chromatic}, Stanley shows that $\phi(X_P) = t^{\sources(P)}$ for any poset $P$.
As a special case, one can show $\phi(\elementaryE_\lambda) = t^{\length(\lambda)}$ by taking $P$ to
be the union of chains of length $\lambda_1$, $\lambda_2$ and so on.
It is now  straightforward see that the sources of $P(\theta)$ exactly
correspond to vertices contributing to $\halfsources(\theta)$, so that $\sources(P(\theta)) = \halfsources(\theta)$.
Putting it all together, we have
\[
\phi(X_\theta) = \phi(X_{P(\theta)}) = t^{\sources(P(\theta))} = t^{\halfsources(\theta)}.
\]
Finally, applying $\phi$ on both sides of \cref{eq:llt-as-orientation-sum};
\begin{align}
\sum_\mu d^\nuvec_\mu(q) \elementaryE_\mu(\xvec) = \sum_{ \substack{\theta:O_\ast(\Gamma_\nuvec)} } 
q^{\asc_\nuvec(\theta)} X_{\theta}
\end{align}
establish the first identity in \cref{eq:llt-e-coefficient-sum}.
The last identity now follows from the fact that $\LLT_\nuvec(\xvec;q+1)$ is symmetric ---
restricting to $n$ variables, and sending color $i$ to color $n+1-i$ turns order-preserving maps
to order-reversing maps and sources to sinks.
\end{proof}

\begin{corollary}
Let $\nuvec$ be a non-circular unit interval graph with some strict corner edges,
and $d^\nuvec_\mu(q)$ defined as in \cref{prop:vert-strip-llt-e-positivity}.
Then
 \begin{align}
 \sum_{\mu} d^\nuvec_\mu(q) = (1+q)^{|\nuvec|}.
\end{align}
where $|\nuvec|$ denotes the number of non-strict edges in $\Gamma_\nuvec$.
\end{corollary}
\begin{proof}
Every orientation of $\Gamma_\nuvec$ is free from ascending cycles.
There are $|\nuvec|$ edges in the graph that may contribute to ascents
and each such edge can independently be chosen to be ascending or not.
\end{proof}

\begin{lemma}
For any $\avec$, the polynomial $\LLT_\avec(x_1,x_2;q+1)$ evaluated in \emph{two variables}
has non-negative coefficients in the $\elementaryE$-basis.
\end{lemma}
\begin{proof}
This is evident from the recursion in \cref{eq:llt2variableRecursion}.
\end{proof}

\begin{corollary}\label{cor:e1ncoeffInLLT}
The coefficient of $\elementaryE_{1^n}(\xvec)$ in $\LLT_\avec(\xvec;q+1)$ is equal to $1$ for all $\avec$.
\end{corollary}
\begin{proof}
The coefficient we seek is given by the sum over all orientations in $O_\ast(\Gamma_\avec)$ such that every vertex is a half-sink.
Due to the definition of half-sinks, it is straightforward to show that there is a unique
orientation of $\Gamma_\avec$ such that every vertex is a half-sink, obtained by reversing all edges of $\Gamma_\avec$.
\end{proof}

\begin{corollary}\label{cor:quasisymmetricPositivity}
In \cite{Stanley95Chromatic}, Stanley shows that $X_P$ in \eqref{eq:posetQSfunc}
expands positively in the Gessel fundamental basis.
It follows that $\LLT_\nuvec(\xvec;q+1)$ is positive in this basis as well with the following expansion:
\[
\LLT_{\nuvec}(\xvec, q+1) = \sum_{\theta: O_*(\Gamma_\nuvec)} q^{\asc_\nuvec(\theta)} \sum_{\pi \in \mathcal{L}(P(\theta),w_\theta) } Q_{D(\pi)},
\]
where $w_\theta$ is a order reversing labeling of $P(\theta)$ and $\pi$ is linear extension of $P$ viewed as a permutation of the labels in $w_\theta$.
\end{corollary}
\emph{Note:} We cannot hope to this all circular ribbon LLT polynomials ---
for example, $\LLT_\nuvec(\xvec,q+1)$ with area sequence $(1,1,1)$ and 
the \emph{weak} inequalities $F(1) \geq F(2) \geq F(3)$ does not expand positively in the fundamental basis.

\begin{proof}
We have that $X_{P(\theta)} = \sum_{\pi \in \mathcal{L}(P(\theta),w_\theta)} Q_{D(\pi)}$, where $D(\pi)$ is the descent set of $\pi$, 
after fixing an order-reversing labeling $w_\theta$ on $P(\theta)$ and 
regarding $\pi$ as a permutation on these labels, \emph{i.e.} the word $w(\pi^{-1})$.
Putting this in \cref{eq:llt-as-orientation-sum} gives the expansion.
\end{proof}

\subsection{Other consequences}

Recall the definition of bounce path in \cref{sec:LLTDiagrams}.
\cref{prop:chromaticAcyclicOrientations} implies the following corollary:
\begin{corollary}
Let $\avec$ be a circular area sequence and let $r$ be the number of complete subgraphs in the bounce path.
If 
\[
\chrom_\avec(\xvec;q) = \sum_\mu c^\avec_\mu(q) \elementaryE_\mu(\xvec),
\text{ then }
\sum_{\substack{\mu \\  \length(\mu)>r}} c^\avec_\mu(q) = 0.
\]
\end{corollary}
\begin{proof}
Each acyclic orientation of $\Gamma_\avec$ contains at most one sink in each of the complete subgraphs of $\Gamma_\avec$
corresponding to parts of the bounce path.
The conclusion now follows from \cref{prop:chromaticAcyclicOrientations}.
\end{proof}

\begin{proposition}
Let $\avec$ be a circular area sequence of length $n$, and suppose one of the two conditions below hold:
\begin{itemize}
 \item $a_i \geq n/2$ for all $i$,
 \item $\max_i a_i = n-1$.
\end{itemize}
Then $\chrom_\avec(\xvec;q) = C^\avec(q) \elementaryE_{(n)}(\xvec)$ for 
some $C^\avec(q)$ with non-negative integer coefficients.
\end{proposition}
\begin{proof}
Note that in these two case, every vertex of $\Gamma_\avec$ is connected with an edge to every other vertex.
It follows that every proper coloring of $\Gamma_\avec$ must use distinct colors for the vertices.
This implies the statement.
\end{proof}

The following conjecture extends and refines \cref{conj:chromatic-e-positivity} to the circular case:
\begin{conjecture}
Let $\avec$ be a circular area sequence and $\chrom_\avec(\xvec;q)$ be
the corresponding chromatic symmetric function.
Then the coefficients $ c_\lambda(q)$ in the expansion 
$\chrom_\avec(\xvec;q) = \sum_\lambda c_\lambda(q) \elementaryE_\lambda(\xvec)$
are palindromic and unimodal polynomials with non-negative integer coefficients.
\end{conjecture}

The following lemma establishes the palindromic property of the $\elementaryE$-coefficients.
\begin{lemma}\label{lem:palindromicity}
The coefficients of $\chrom_\avec(\xvec;q)$ in the $\elementaryE$-basis are palindromic.
\end{lemma}
\begin{proof}
Suppose $n$ is the number of vertices of $\Gamma_\avec$.
The coefficients are palindromic, because this holds in any basis as long as the symmetry is 
about the same degree, which in this case is $|\avec|/2$. 
To prove palindromicity in the monomial basis, first restrict to $n$ variables and 
consider the map on colorings that send color $i$ to color $n+1-i$. 
This map sends ascending edges to non-ascending 
edges and vice versa, and there $|\avec|$ edges in total.
\end{proof}

\subsection{Path, cycle and the complete graph: Chromatic case.}

The results in the following theorem has also been proved using 
different methods in \cite{ShareshianWachs2010,Ellzey2016}. 
\begin{theorem}\label{thm:generatingFunctionsPathCycle}
Let $P_n$ and $C_n$ denote the line and cycle graph on $n$ vertices. Then
\[
\sum_n \chrom_{P_n}(\xvec;q)z^n =  
\frac{ \sum_{i\geq 0} \elementaryE_i(\xvec) z^i  }{ 1 - q \sum_{i \geq 2} [i-1]_q \elementaryE_i(\xvec) z^i }
\]
and
\[
\sum_n \chrom_{C_n}(\xvec;q) = 
\frac{ q\sum_{i\geq 2} i [i-1]_q \elementaryE_i(\xvec) z^i  }{ 1 - q \sum_{i \geq 2} [i-1]_q \elementaryE_i(\xvec) z^i }.
\]
\end{theorem}
\begin{proof}
We will first show the formula for the path graph $P_n$. 
The proof is by induction on $n$ and the number $m$ of variables appearing, $\xvec_m\coloneqq (x_1,\ldots,x_m)$.

\textbf{Case $m=1$:} We have $\chrom_{P_n}(x_1;q) = x_1=e_1(x_1)$ when $n=1$ and $0$ otherwise. 
We have that $\elementaryE_{\mu}(x_1)=0$ unless $\mu =1^k$ 
and any acyclic orientation would have sectors of size at least $2$. 
Thus if $n>1$ there would be no terms $e_{1^k}$ in the right hand side and the formula holds.
\medskip 

\textbf{Case $m=2$:} There are only two proper colorings of $P_n$ --- either alternating $121\dotsc$ or 
alternating $212\dotsc$, which give
\begin{align*}
\chrom_{P_n}(x_1,x_2;q) =
\begin{cases}
q^k (x_1^{k+1}x_2^k + x_1^k x_2^{k+1})=q^ke_{2^k1}(x_1,x_2) & \text{ if } n=2k+1 \\  
x_1^kx_2^k (q^k + q^{k+1}) = (q^k + q^{k+1})e_{2^k}(x_1,x_2) &\text{ if } n=2k.
\end{cases}
\end{align*}
Since $e_\mu(x_1,x_2)=0$ if $\mu_1 >2$, the acyclic orientations can only have sectors of size $1$ or $2$, 
and thus the vertices are alternating sinks and sources. The formulas match again.

\textbf{Case $m\geq 3$:} In a proper coloring of $P_n$, let the vertices colored $m$ be at positions 
$\alpha_1, \alpha_1+\alpha_2,\ldots, \alpha_1+\cdots+\alpha_k=r \leq n$, where $\alpha$ 
is a composition of some $r\leq n$ and $\alpha_i>1$ for $i>1$. 
The total number of inversions introduced by the color $m$ is $k$ if $r<n$ or $k-1$ if $r=n$. 
For brevity, let $X_n(\xvec)\coloneqq \chrom_{P_n}(\xvec;q)$.  We then have
\begin{equation}\label{eq:pathRecurrence}
\begin{aligned}
\chrom_{P_n}(\xvec_m;q) = & \sum_{r<n,k, |\alpha| = r} q^k x_m^k 
 \left[\prod_{i=1}^k  X_{\alpha_i -1}(\xvec_{m-1}) \right] X_{n-r}(\xvec_{m-1})  \\
& + \sum_{k, |\alpha|= n} q^{k-1} x_m^k \prod_{i=1}^k X_{\alpha_i-1}(\xvec_{m-1}),
\end{aligned}
\end{equation}
where the two sums represent colorings where vertex $n$ has color either less than $m$, or $m$, respectively.

Now set 
\[
H_m(z) \coloneqq \sum_{n=0}^{\infty} \chrom_{P_n}(\xvec_m;q)z^n 
  = 1 + \elementaryE(\xvec_m)z + \sum_{n \geq 2} \chrom_{P_2}(\xvec_m;q)z^n.
\]
The recursive formula \eqref{eq:pathRecurrence} can be written as
\begin{align*}
H_m(z) &= \sum_{k=0}^{\infty} (qx_m)^k z^k H_{m-1}(z)(H_{m-1}(z)-1)^{k} + \sum_{k\geq 1} q^{k-1}x_m^k z^k H_{m-1}(z)(H_{m-1}(z)-1)^{k-1}\\
&= H_{m-1}(z) \frac{1+zx_m}{1-qzx_mH_{m-1}(z) + qzx_m}.
\end{align*}
We now prove the generating function version of the formula.
Let $F(\xvec;z) \coloneqq \sum_{i \geq 0} \elementaryE_i(\xvec)z^i$, and $F_m(z)\coloneqq F(\xvec_m;z) = (x_mz+1) F_{m-1}(z)$. 
Our goal is now to show that
\begin{align*}
H_m (z)&= \frac{ \sum_{i\geq 0} \elementaryE_i(\xvec_m) z^i  }{ 1 - q \sum_{i \geq 2} [i-1]_q \elementaryE_i(\xvec_m) z^i }\\
&= \frac{ F(\xvec_m;z)}{ 1 - \frac{1}{q-1} \left(F(\xvec_m;qz) -1 -\elementaryE_1(\xvec_m)qz - q F(\xvec_m;z) +q + q \elementaryE_1(\xvec_m)z\right)}\\
&=(q-1)\frac{ F_m(z)}{ -F_m(qz)  + q F_m(z)}.
\end{align*}
By induction on $m$ we have:
\begin{align*}
H_m(z) &= H_{m-1}(z) \frac{1+zx_m}{1-qzx_mH_{m-1}(z) + qzx_m}  \\
&= (q-1)\frac{ F_{m-1}(z)(1+zx_m)}{ (-F_{m-1}(qz)  + q F_{m-1}(z) ) \left(1 +qzx_m -qzx_m (q-1)\frac{ F_{m-1}(z)}{ -F_{m-1}(qz)  + q F_{m-1}(z) }\right)  } \\
&=\frac{(q-1)F_m(z)}{- F_{m-1}(qz)  + q F_{m-1}(z)  - qzx_m F_{m-1}(qz)  + q^2zx_m F_{m-1}(z) -q^2zx_m F_{m-1}(z) + qzx_mF_{m-1}(z)}\\
&= \frac{(q-1) F_m(z)}{ - (1+zqx_m)F_{m-1}(qz) + q(1+zx_m)F_{m-1}(z) } \\
&= (q-1)\frac{F_m(z)}{-F_m(qz)+qF_m(z)}
\end{align*}
which is what we wanted to prove.

\medskip

The formula for $C_n$ is proved in  a similar fashion where we use the formula for the path graph.
For a coloring of the cycle, either no vertex is colored $m$, or $k$ vertices with color $m$ are dividing the 
cycle into sectors of sizes $\alpha_i>1$ for $i=1,\dotsc,k$, which are themselves path graphs of 
length $\alpha_i-1$ in the colors $\xvec_{m-1}$. Thus
\begin{align*}
\chrom_{C_n}(\xvec_m;q) = \chrom_{C_n}(\xvec_{m-1},q) 
 + \sum_{k\geq 1} q^kx_m^k \sum_{ |\alpha|=n }\alpha_1 \prod_{i=1}^k X_{\alpha_i-1}(\xvec_{m-1},q). 
\end{align*}
Note that since we have a cycle, we need to choose where in the first sector vertex $1$ appears.
This explains the $\alpha_1$ in the formula.
\medskip 

If we let $H^c_m(z) \coloneqq  \sum_{n} \chrom_{C_n}(\xvec_m,q)z^n$, then
\begin{align*}
H^c_m(z) &= H^c_{m-1}(z) + \sum_{k\geq 1} q^kx_m^k z^k  \left(\frac{\partial  z(H_{m-1}(z)-1) }{\partial z}\right) (H_{m-1}(z)-1)^{k-1} \\
&= H^c_{m-1}(z) + \frac{qx_mz \left(z \frac{\partial H_{m-1}(z) }{\partial z} +H_{m-1}(z)-1\right) }{1 -qx_mz(H_{m-1}(z)-1)}.
\end{align*}
We need to show the following formula, which we prove by induction on $m$:
\begin{align*}
 H^c_m(z) = \frac{ q\sum_{i\geq 2} i [i-1]_q \elementaryE_i z^i  }{ 1 - q \sum_{i \geq 2} [i-1]_q \elementaryE_i z^i }
 =\frac{ zq(F'_m(qz) - F'_m(z))}{-F_m(qz) +qF_m(z)}.
\end{align*}

Here we noted that the denominator is the same as for $H_m$ and the numerator can be written as
\[ 
z \frac{\partial}{\partial z} q \sum_{i\geq 2} [i-1]_q\elementaryE_i z^i = \frac{z}{q-1} \frac{\partial}{\partial z} ( F_m(qz) +q-1 -qF_m(z) )=\frac{z}{q-1}  ( F_m'(qz)q -qF_m'(z) ).
\]
We also calculate that
\[
\frac{\partial H_{m}(z)}{\partial z} = (q-1)\frac{-F'_m(z)F_m(qz)+qF_m(z)F'_m(qz)}{(-F_m(qz)+qF_m(z))^2},
\quad H_{m} - 1 = \frac{F_m(qz) - F_m(z) }{qF_m(z) - F_m(qz)},
\]
and
\[
1-qx_mz(H_{m-1}(z)-1) = \frac{qF_m(z) - F_m(qz) }{qF_{m-1}(z) - F_{m-1}(qz)}.
\]
Using these identities in the recursion for $H^c$, the induction hypothesis and the fact that 
$F'_m(u) = (1+x_mu)F'_{m-1}(u) + x_mF_m(u)$, we get
\begin{align*}
H_m^c(z) &= \frac{zq(F'_{m-1}(qz) -F_{m-1}'(z))}{qF_{m-1}(z) -F_{m-1}(qz)} \\
&\phantom{=}+ qx_mz\frac{ z(q-1)( qF_{m-1}(z)F'_{m-1}(qz) -F_{m-1}'(z)F_{m-1}(qz)) }{(qF_{m-1}(z)-F_{m-1}(qz))(qF_m(z)-F_m(qz))}\\ 
&\phantom{=}+ qx_mz\frac{ (F_{m-1}(qz)-F_{m-1}(z))}{(qF_m(z)-F_m(qz))} \\
 &= zq  \frac{ ((1+qzx_m)F'_{m-1}(qz) -(1+zx_m)F'_{m-1}(z)) (qF_{m-1}(z) - F_{m-1}(qz) )}{(qF_{m-1}(z)-F_{m-1}(qz))(qF_m(z)-F_m(qz))} \\
&\phantom{=}+ qzx_m\frac{F_{m-1}(qz)-F_{m-1}(z)}{qF_m(z)-F_m(qz)}\\
&= zq \frac{ x_m F_{m-1}(qz) -x_m F_{m-1}(z) + F'_m(qz)-x_mF_{m-1}(qz) -F'_m(z)+x_mF_{m-1}(z)}{qF_m(z)-F_m(qz)}\\
&= zq\frac{F_m'(qz) - F_m'(z)}{qF_m(z)-F_m(qz)}
\end{align*}
as desired.
\end{proof}

In \cite[Prop.\ 5.4]{Stanley95Chromatic}, Stanley consider the $\elementaryE$-expansion of the cycle graphs,
\emph{i.e.,} $\avec = (1,1,\dotsc,1)$ and show that the expansion is positive in the case $q=1$,
similar to how the above formulas imply $\elementaryE$-positivity.
However, no combinatorial interpretation of the $\elementaryE$-coefficients is given.
In the following theorem, we present such a combinatorial interpretation. This interpretation also appears in~\cite{Ellzey2016}.
\begin{theorem}\label{thm:chromaticCycleGraphEexp}
Let $P_n$, $C_n$, $K_n$ and $B_n$ denote the line graph, the cycle graph,
the complete unit interval graph and the complete circular unit arc digraph on $n$ vertices.
Let $\Gamma$ be any disjoint union of such graphs. Then
\begin{align}\label{eq:path-cycle-formula}
\chrom_{\Gamma}(\xvec;q) =\!\!\!\!\! \sum_{ \theta : AO(\Gamma) } q^{\asc_{\Gamma}(\theta)} \elementaryE_{\mu(\theta)}(\xvec)
\end{align}
where $\mu(\theta)$ is the sizes of the \emph{circle sectors} when using the \emph{sinks} of $\theta$ as dividers.
\end{theorem}

Before proving this theorem, we give an example on how to find $\mu(\theta)$ of an orientation of $P_n$ or $C_n$.
For orientations $\theta$ of $K_n$ and $B_n$, $\mu(\theta)=(n)$ for all orientations,
since orientations on these graphs have unique sinks.
\begin{example}
In the following two figures,
we have an orientation of $P_8$ and $C_8$, respectively.
The sources are marked with a bar and the gray vertices are the sinks. 
For each sink, we have an associated circle sector, consisting 
of the sinks and the cyclically following non-sinks clockwise.
The sectors in both orientations are $\{2,1,8\}$, $\{4,3\}$ and $\{7,6,5\}$, giving the shape $\mu(A)=332$.

Note that in the first orientation, the ``virtual'' edge $8$--$1$ does not have an orientation,
and we have that the number of ascents is $3$, as there are $3$ edges 
oriented in the same direction as the underlying orientation of the path. 
In the second orientation, there are $4$ ascents.

\begin{center}
\includegraphics[width=0.8\textwidth]{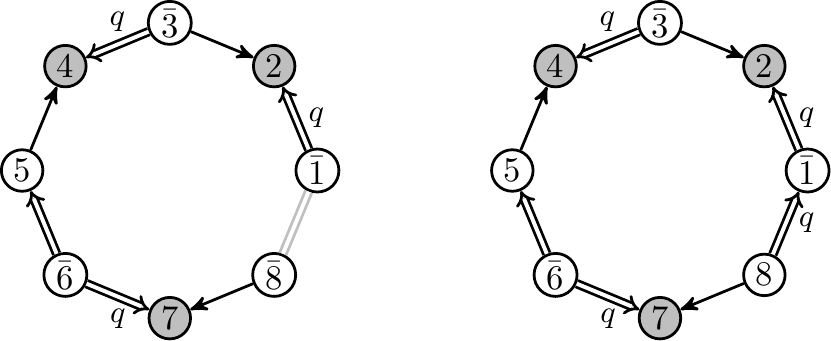}
\end{center}

\end{example}
\begin{proof}
Let $\chrom_{C_n}(\xvec;q)$ denote the chromatic symmetric function for a cycle with $n$
vertices. Recall from \cref{thm:generatingFunctionsPathCycle} the generating function identity
\begin{align*}
 \sum_{z \geq 0}\chrom_{C_n}(\xvec;q) z^n = 
 \frac{ q\sum_{i\geq 2} i [i-1]_q \elementaryE_i z^i  }{ 1 - q \sum_{i \geq 2} [i-1]_q \elementaryE_i z^i }.
\end{align*}
We will show that the formula in \eqref{eq:path-cycle-formula} satisfies this generating function.

An acyclic orientation of $C_n$ partitions $C_n$ into sectors
determined by the sinks. A sector consists of the sink $s$
and all vertices clockwise from $s$ until the next sink. 
Note that every sector has size at least two,
and that there is exactly one source in each sector.

\begin{itemize}
\item The numerator $q \sum_{i\geq 2} i [i-1]_q \elementaryE_i z^i$ describes
constructing the unique sector containing vertex $1$.
After picking the size $i$, there are $i$ ways to assign 
which of the $i$ vertices in the sector that has label $1$.
The $[i-1]_q$ factor determines the position of the source --- there are only $i-1$
choices since we cannot choose the sink.
The extra $q$ comes from the fact that the sink always contributes with an ascending edge.

\item The denominator now simply adds more sectors after the initial one,
using a similar reasoning as for constructing the first sector, where each sink 
has an ascending edge. Edges between sectors are always descending.
\end{itemize}

\medskip 

Let $\chrom_{P_n}(\xvec;q)$ denote the chromatic symmetric function for a path with $n$ vertices.
We previously proved
\begin{align*}
 \sum_{z \geq 0}\chrom_{P_n}(\xvec;q) z^n &= 
 \frac{ \sum_{i\geq 0} \elementaryE_i z^i  }{ 1 - q \sum_{i \geq 2} [i-1]_q \elementaryE_i z^i } \\
 &=
 \frac{ \sum_{i\geq 1} \elementaryE_i z^i  }{ 1 - q \sum_{i \geq 2} [i-1]_q \elementaryE_i z^i }+
 \frac{ 1  }{ 1 - q \sum_{i \geq 2} [i-1]_q \elementaryE_i z^i }.
\end{align*}

Since vertex $1$ is an end-point of the path, it can either be a sink or a source.
These two cases correspond to the two terms we indicated above.

\noindent
\textbf{Vertex $1$ is a sink.} 
There is a unique vertex $v$ such that 
\[
1 \quad \text{ --- } \quad n \rightarrow (n-1) \rightarrow \dotsb \rightarrow v
\]
is a sector in the orientation.
There is no $q$ here, since all these edges are descending, 
and we are not yet sure if there is a second sink following $v$, since we could have $v=1$.
The denominator now adds additional sectors with unique sources just as in the cycle graph.

\noindent
\textbf{Vertex $1$ is a source.} The second term in the expression is given by
\[
1 + \left( q \sum_{i \geq 2} [i-1]_q \elementaryE_i z^i \right) + \left( q \sum_{i \geq 2} [i-1]_q \elementaryE_i z^i \right) ^2 + \dotsb
\]
We need to show that constructing a sector of size $i$ with $1$ as a source,
correspond to the expression $q [i-1]_q \elementaryE_i z^i$.
Note that $1$ belongs to some sector
\[
 s \leftarrow (s-1) \leftarrow \dotsb \leftarrow 1 \quad \text{ --- } 
 \quad n \rightarrow (n-1) \rightarrow \dotsb \rightarrow v
\]
with $s\geq 2$ being the sink. The factor $[i-1]_q$
correspond in this case to the choice of $s$ (which determines $v$ uniquely since the size must be $i$)
and thus the number of ascending edges in the sector.
We are guaranteed to have at least one ascending edge, $s \leftarrow (s-1)$,
which account for the extra $q$.
As before, the remaining part of the orientation is created sector by sector.

\medskip

The graphs $K_n$ and $B_n$ only allow colorings with distinct colors.
Each such coloring $F$ induce an acyclic orientation $\theta$, such that $\asc \theta = \asc F$.

Finally, the case when $\Gamma$ is a disjoint union of smaller graphs follows from the
fact that the $e$-basis is multiplicative and the fact that $\asc$ and $\mu(\theta)$
are additive on disjoint graph components.
\end{proof}

The following propositions explicitly give the $\elementaryE$-coefficients
in case of the path and cycle graph.

\begin{proposition}
Let $P_n$ be the path graph with $n$ vertices, and let $\mu$ be a partition with $k$ parts.
Then
\begin{align}\label{eq:path-graph-e-coeffs}
[\elementaryE_{\mu}]\chrom_{P_n}(\xvec;q)  = 
\frac{
k! q^{k-1}  \elementaryE_{k-1}([\nu_1]_q,\dotsc,[\nu_{k}]_q)
}{m_1(\mu)!\dotsm m_n(\mu)!} 
+
\frac{
k!q^{k}\elementaryE_k([\nu_1]_q,\dotsc, [\nu_{k}]_q)
}{m_1(\mu)!\dotsm m_n(\mu)!} 
\end{align}
where $\nu_i = \mu_i-1$ for $i=1,\dotsc,k$. 
\end{proposition}
\begin{proof}
We use the same model and sub-cases as in \cref{thm:chromaticCycleGraphEexp}.

The first fraction treats vertex $1$ as a sink in a \emph{special} sector
and $1$ does not contribute with a descending edge.
\begin{itemize}
\item The $k!/m_1(\mu)!\dotsm m_n(\mu)!$ accounts for permuting sector sizes.
\item The expression $\elementaryE_{k-1}([\nu_1]_q,\dotsc,[\nu_{k}]_q)$
corresponds to deciding which sector is \emph{special}, and then placing the sources within
the other $k-1$ sectors. Each of these sectors contain one descending edge next to the sink.
\end{itemize}

The second fraction is the weighted count of the orientations 
with sector sizes given by $\mu$ and vertex $1$ is a source.
\begin{itemize}
\item The $k!/m_1(\mu)!\dotsm m_n(\mu)!$ again accounts for all permutations of the sector sizes,
where the first sector is the one containing vertex $1$.
\item The $q^k$ accounts for the fact that each sink now contributes with exactly one ascending edge.
\item Finally, $\elementaryE_k([\nu_1]_q,\dotsc, [\nu_{k}]_q)$ describes
the placement of the unique source within each sector,
or in the case of sector $1$, starting and ending vertices.
\end{itemize}
\end{proof}

\begin{proposition}
Let $C_n$ be the cycle graph with $n$ vertices, and let $\mu$ be a partition with $k$ parts.
Then
\begin{align}
[\elementaryE_{\mu}]\chrom_{C_n}(\xvec;q)  = 
\sum_{j \in \mu}
\frac{(k-1)!q^{k} \cdot j \cdot 
 \elementaryE_{k-1}([\mu^j_1-1]_q,\dotsc, [\mu^j_{k-1}-1]_q)
}{m_1(\mu^j)!\dotsm m_n(\mu^j)!},
\end{align}
where the sum runs over all \emph{different} parts of $\mu$,
and $\mu^j$ is the partition obtained from $\mu$ by removing 
one part of size $j$. For example, $\mu = 743321$ gives
\[
 \mu^3 = 74321 \qquad \text{ and }\qquad \mu^2 = 74331. 
\]
\end{proposition}
\begin{proof}
The proof follows a similar reasoning as for the path case.
The sum is over all possible sizes of sectors that contain the vertex $1$.
Within this sector, there are $j$ ways to choose which vertex has label $1$.
The remaining $k-1$ sectors can then be permuted freely, just as in the path case. 
Note that each sector now contains a sink with an ascending edge --- explaining the $q^k$.
\end{proof}

\begin{remark}
In \eqref{eq:path-cycle-formula} we can interchange 
replace \emph{sinks} with \emph{sources} and count descending edges instead of ascending edges,
as there is a bijection on acyclic orientations given by reversing all edges.
Since we proved that the $\elementaryE$-coefficients are palindromic in
\cref{lem:palindromicity}, we can in fact use any of the four combination of ascending/descending edges 
and sinks/sources.
However,only the combinations $(\inv,\text{sinks})$ and $(\asc,\text{sources})$
allow us to give interpretations of the individual terms in \eqref{eq:path-graph-e-coeffs}.
\end{remark}

Finally, we note that a recent preprint, \cite{DahlBergWilligenburg2017},
give an algebraic proof of $\elementaryE$-positivity of another family of graphs,
each such graph consisting of a complete graph glued together with a chain.

\subsection{Path, cycle and the complete graph: LLT case}

There is also an analogue of \cref{thm:chromaticCycleGraphEexp} in the LLT case:
\begin{proposition}\label{prop:lltCycleGraphEexp}
Let $\avec$ be any of the graphs
\begin{itemize}
 \item $(1,1,\dotsc,1,1)$,
 \item $(1,1,\dotsc,1,0)$,
 \item $(n-1,n-2,\dotsc,0)$,
\end{itemize}
\emph{i.e.,} the line graph,
the circuit or the acyclic complete graph on $n$ vertices. Then
\begin{align}\label{eq:completeGraphLLTDef1}
\LLT_\avec(\xvec;q+1) =\sum_{\theta:O_\ast(\Gamma_\avec)} q^{\asc(\theta)} \elementaryE_{\mu(\theta)}(\xvec).
\end{align}
where $\mu(\theta)$ is the sizes of the circle sectors when using the half-sinks of $\theta$ as dividers.
\end{proposition}
\begin{proof}

We consider the formula given in \cref{eq:XthetaDef}. In this case
of a cycle graph, it is fairly straightforward to see that
the poset $P(\theta)$ is a disjoint union of chains with lengths given by $\mu(\theta)$,
and then that $X_\theta = \elementaryE_{\mu(\theta)}(\xvec)$.

Note that for the path graph $\avec = (1,\dotsc,1,0)$, vertex $n$ is always a half-sink.
This prevents circle sectors to ``wrap around'',
and a similar reasoning as in the cycle graph case shows that 
again, $X_\theta = \elementaryE_{\mu(\theta)}(\xvec)$.

\medskip

The third case is more involved and requires several steps.
First, define
\begin{align}\label{eq:completeGraphLLTDef2}
\hat{\LLT}_{K_n}(\xvec;q+1)  \coloneqq \sum_{\theta:O_\ast(\Gamma_\avec)} q^{\asc(\theta)} \elementaryE_{\mu(\theta)}(\xvec).
\end{align}
Our goal is to show that $\LLT_{K_n}(\xvec;q)  = \hat{\LLT}_{K_n}(\xvec;q)$.
First, it is fairly straightforward from the definition in \eqref{eq:completeGraphLLTDef2}
to obtain the recurrence
\begin{align}\label{eq:lltCompleteGraphRecurrence}
 \hat{\LLT}_{K_n}(\xvec;q+1) = \sum_{i=0}^{n-1} \hat{\LLT}_{K_i}(\xvec;q+1) e_{n-i}(\xvec)  \prod_{k=i+1}^{n-1} \left[ (q+1)^k -1 \right]
 ,\; \LLT_{K_0}(\xvec;q+1) = 1.
\end{align}
Basically, every orientation of $K_n$, can be obtained by first
orientating the bottom $i$ rows, followed by making row $i+1$ from the bottom the top-most half-sink,
thus adding a sector of size $n-i$. This also forces the orientations of all edges in row $i+1$.
Finally we need to orient the edges in the remaining $n-i-1$ rows, while avoiding creating more half-sinks.
The only configuration we need to avoid is having all edges ``pointing right''
in some row. This explains the product in the formula.

Let
\[
F(x;q) \coloneqq \sum_{n \geq 0} \frac{ \hat{\LLT}_{K_n}(\xvec;q) }{ (1-q^n)\dotsm (1-q^2)(1-q) }.
\]
Applying the recurrence \eqref{eq:lltCompleteGraphRecurrence} and noting that $\hat{\LLT}_{K_n}(\xvec;q)$ is a homogeneous polynomial in $\xvec$ of degree $n$, we have the following
\begin{align*}
F(\xvec;q) -F(q\xvec;q) &= \sum_{n \geq 1} \frac{ \sum_{i=0}^{n-1} \left( \hat{\LLT}_{K_i}(\xvec;q) e_{n-i}(\xvec)  - \hat{\LLT}_{K_i}(q\xvec;q) e_{n-i}(q\xvec) \right)\prod_{k=i+1}^{n-1} \left[ q^k -1 \right] }{ (1-q^n)\dotsm (1-q^2)(1-q) }\\
&=\sum_{n \geq 1} \frac{ \sum_{i=0}^{n-1}  \hat{\LLT}_{K_i}(\xvec;q) e_{n-i}(\xvec)  (1-q^n) \prod_{k=i+1}^{n-1} \left[ q^k -1 \right] }{ (1-q^n)\dotsm (1-q^2)(1-q) }\\
&=\sum_{i \geq 0} \frac{\hat{\LLT}_{K_i}(\xvec;q)}{(1-q^i)\dotsm (1-q)} \sum_{r \geq 1} (-1)^{r-1} e_r(\xvec) = F(\xvec;q)\left( 1-\prod_j (1-x_j) \right).
\end{align*}
Solving for $F(\xvec;q)$ in terms of $F(q\xvec;q)$ and iterating we get
\[
F(\xvec;q) = \frac{ F(q\xvec;q) }{\prod_j (1-x_j)} = \cdots = \prod_{r=1}^m \prod_{j} \frac{1}{1-x_jq^{r-1}} F(q^m\xvec;q),
\]
leading to the generating function identity
\[
\sum_{n\geq 0} \frac{ \hat{\LLT}_{K_n}(\xvec;q) }{ (1-q^n)\dotsm (1-q^2)(1-q) } = \prod_{i,j \geq 0} \frac{1}{1-x_i q^j}.
\]
The right hand side can be interpreted as a specialization of the Cauchy identity, see \emph{e.g.,} \cite{StanleyEC2}:
\begin{align}
 \prod_{i,j \geq 0} \frac{1}{1 - x_i y_j}  = \sum_{\lambda} \schurS_\lambda(\xvec)\schurS_\lambda(\yvec) \tag{Cauchy Identity}
\end{align}
Hence,
\begin{align}
 \sum_{n\geq 0} \frac{ \hat{\LLT}_{K_n}(\xvec;q) }{ (1-q^n)\dotsm (1-q^2)(1-q) }
 &= \sum_{\lambda}  \schurS_\lambda(\xvec)\schurS_\lambda(1,q,q^2,\dotsc) \nonumber \\
 &= \sum_{\lambda}  \schurS_\lambda(\xvec) \frac{ \sum_{T\in \SYT(\lambda)} q^{\comaj(T)} }{(1-q^n)\dotsm (1-q^2)(1-q)}
 \label{eq:lltCompleteGraph1}
\end{align}
where the second equality is due to \cite[Prop. 7.19.11]{StanleyEC2}.
As a side note, the right hand side above is the Frobenius series of $\setC[x_1,\dotsc,x_n]$
under the usual $\symS_n$ action.

\medskip
On the other side of the identity we want to prove, we have by definition that
\begin{align}\label{eq:lltCompleteGraph2}
\LLT_{K_n}(\xvec;q) = \sum_{w \in \setN^n} x_{w_1}\dotsm x_{w_n} q^{\inv(w)}
= \sum_{T\in \SYT(n)} q^{\cocharge(T)} \schurS_{\lambda(T)}(\xvec)
\end{align}
where the middle sum is over all words of length $n$ with letters in $\setN$,
and the second equality is an identity that follows from the Robinson--Schensted--Knuth correspondence.
By comparing Schur coefficients in \eqref{eq:lltCompleteGraph1}
and \eqref{eq:lltCompleteGraph2}, the identity now follows (see \cite[p. 16]{qtCatalanBook})
from the fact that $\cocharge(T) = \comaj(T)$.
\end{proof}

As a small remark, the series
\[
 \sum_{n\geq 0} \LLT_{K_n}(\xvec;q) = \sum_{\lambda } \schurS_\lambda(\xvec) \sum_{T\in \SYT(\lambda)} q^{\comaj(T)}
\]
is the Frobenius series of the ring of \emph{contravariants},
that is, $\setC[x_1,\dotsc,x_n]/\langle \elementaryE_1,\dotsc, \elementaryE_n \rangle$,
see the first chapter in \cite{qtCatalanBook}.

\begin{corollary}
Let $P_n$ and $C_n$ denote the path and the cycle graph on $n$ vertices, respectively. Then
\begin{align*}
 \sum_{n \geq 0} \LLT_{P_n}(\xvec;q+1)z^n = \frac{1}{1 - \sum_{i\geq 1} q^{i-1} z^i e_i },
\end{align*}
and
 \begin{align*}
 \sum_{n \geq 0} \LLT_{C_n}(\xvec;q+1)z^n = \frac{\sum_{i\geq 1} i q^{i-1} z^i e_i}{1 - \sum_{i\geq 1} q^{i-1} z^i e_i }.
\end{align*}
\end{corollary}
\begin{proof}
These identities are straightforward to prove using the combinatorial formula in \cref{prop:lltCycleGraphEexp}.
\end{proof}

We end this section with a conjecture indicated by computer experiments:
\begin{conjecture}
Let $\Gamma$ be a circular unit arc digraph, and let $H$ be a graph obtained from $\Gamma$ by marking $k$ corner edges strict.
Then
\[
 \LLT_\Gamma(\xvec;q+1) - q^k \LLT_H(\xvec;q+1)
\]
is $\elementaryE$-positive. 
\end{conjecture}

\section{Expansion of unicellular LLT polynomials in power-sum basis}

In 2015, C.~Athanasiadis gave the $\psumP$-expansion of the chromatic symmetric functions, associated with unit interval graphs, see \cite{Athanasiadis15}.
By using this expansion together with properties of \emph{plethysitc substitution}, we obtain a
combinatorial formula for the $\psumP$-expansion of unicellular LLT polynomials.
The relation in question is from \cite[Prop. 3.4]{CarlssonMellit2015}:
\begin{lemma}\label{lem:plethRelation}
Let $\avec$ be a unit interval graph. Then
\begin{align} \label{eq:plethRelation}
 (q-1)^{-n} \LLT_\avec[\xvec(q-1);q] =  \chrom_\avec(\xvec;q),
\end{align}
where the bracket denotes a plethystic substitution.
\end{lemma}
This plethystic relation does not extend to the circular case,
where something more involved happens in that case.

\subsection{Unit interval case}

\begin{theorem}[\cite{Athanasiadis15}]
Let $\Gamma_\avec$ be a non-circular unit interval graph.
Then
\begin{align}
\omega \chrom_\avec(\xvec;q) = \sum_{\mu}  c_{\avec,\mu}(q) \frac{\psumP_\mu(\xvec)}{z_\mu}
\end{align}
where $c_{\avec,\mu}(q)$ is a unimodal and palindromic polynomial with non-negative integer coefficients.
In fact,
\[
c_{\avec,\mu}(q) = [\mu_1]_q\dotsm [\mu_k]_q \sum_{\substack{\pi \in \symS_n \\  \pi \text{ $\mu$-admissible} }} q^{\asc_{\avec}(\pi^{-1})},
\]
where $\pi$ is $\mu$-admissible if the following holds: partition $\pi$ into contiguous blocks of size $\mu_i$ ---
that is, the first $\mu_1$ letters of $\pi$ constitute the first block, the next $\mu_2$ letters the second block, and so on.
Each such block $[a_1,\dotsc,a_k]$ is admissible if
\begin{itemize}
\item $a_{i} \leftarrow a_{i+1}$ is never in $P_\avec$ (no $P$-descents).
\item $a_i < a_k$ for $1\leq i < k$.
\end{itemize}
Here, $P_\avec$ is the poset in \cref{ss:posetInterpretation}
\end{theorem}

\bigskip

We now show that a similar statement holds for the LLT polynomials 
(a variant of this is mentioned in \cite{HaglundWilson2017}).
\begin{theorem}\label{thm:lltPPositive}
Let $\Gamma_\avec$ be a non-circular unit interval graph. Then
\begin{align}
\omega\LLT_\avec(\xvec;q+1) = \sum_\lambda q^{n-\length(\lambda)}
\left( \sum_{\substack{\pi \in \symS_n \\  \pi \text{ $\lambda$-admissible} }} (q+1)^{\asc_{\avec}(\pi^{-1})}  \right)\frac{\psumP_\lambda(\xvec)}{z_\lambda}
\end{align}
\end{theorem}
\begin{proof}
From \cref{lem:plethRelation}, it follows that
\begin{align}
\LLT_\avec(\xvec;q) = (q-1)^{n}\chrom_\avec\left[\frac{\xvec}{(q-1)};q\right].
\end{align}
Recall that $\omega \psumP_k(\xvec) = (-1)^{k-1} \psumP_k(\xvec)$, and that
$\psumP_k\left[\xvec/(q-1)\right] = (q^k-1)^{-1}\psumP_k(\xvec)$,
so it is clear that $\omega$ commutes with this type of plethystic substitution.
We have
\begin{align}
\omega\LLT_\avec(\xvec;q) &= (q-1)^{n} \omega\chrom_\avec\left[\frac{\xvec}{q-1};q\right] \nonumber \\
&= (q-1)^{n} \sum_\lambda c_{\avec,\lambda}(q) \frac{1}{z_\lambda} \psumP_\lambda\left[\frac{\xvec}{q-1}\right] \label{eq:lltRelation} \\
&= (q-1)^{n} \sum_\lambda \left(\prod_{i=1}^{\length(\lambda)} \frac{[\lambda_i]_q}{q^{\lambda_i}-1}  \right)
\left( \sum_{\substack{\pi \in \symS_n \\  \pi \text{ $\lambda$-admissible} }} q^{\asc_{\avec}(\pi^{-1})}  \right) \frac{\psumP_\lambda(\xvec)}{z_\lambda} \nonumber \\
&= \sum_\lambda (q-1)^{n-\length(\lambda)}
\left( \sum_{\substack{\pi \in \symS_n \\  \pi \text{ $\lambda$-admissible} }} q^{\asc_{\avec}(\pi^{-1})}  \right)\frac{\psumP_\lambda(\xvec)}{z_\lambda} \nonumber 
\end{align}
Replacing $q$ with $q+1$ now gives the expansion
\begin{align*}
\omega\LLT_\avec(\xvec;q+1) = \sum_\lambda q^{n-\length(\lambda)}
\left( \sum_{\substack{\pi \in \symS_n \\  \pi \text{ $\lambda$-admissible} }} (q+1)^{\asc_{\avec}(\pi^{-1})}  \right)\frac{\psumP_\lambda(\xvec)}{z_\lambda}.
\end{align*}
\end{proof}

\subsection{Circular case}

Now we let $\avec$ be any \emph{circular} area sequence and
define $\hat{c}_{\avec,\lambda}(q)$ via the relation
\begin{align}\label{eq:lltRelation2}
\omega\LLT_\avec(\xvec;q) &= \sum_\lambda \hat{c}_{\avec,\lambda}(q) \left(\prod_{i=1}^{\length(\lambda)} \frac{(q-1)^{\lambda_i}}{q^{\lambda_i}-1}  \right)  \frac{\psumP_\lambda(\xvec)}{z_\lambda}.
\end{align}
Remember that we previously defined
\[
\omega \chrom_\avec(\xvec;q) = \sum_\lambda c_{\avec,\lambda}(q) \frac{\psumP_\lambda(\xvec)}{z_\lambda},
\]
and that \eqref{eq:lltRelation} implies that $c_{\avec,\lambda}(q) = \hat{c}_{\avec,\lambda}(q)$
whenever $\avec$ is non-circular. In \cite{Ellzey2016}, a combinatorial interpretation is given for the $c_{\avec,\lambda}(q)$,
thus giving the $\psumP$-expansion of the chromatic quasisymmetric functions in the circular setting.

\begin{conjecture}
The coefficients $\hat{c}_{\avec,\lambda}(q)$ are unimodal polynomials in $q$ with non-negative integer coefficients.
Furthermore, the difference
\[
 \hat{c}_{\avec,\lambda}(q) - c_{\avec,\lambda}(q)
\]
have non-negative coefficients.
\end{conjecture}
Note that in the non-circular case, the coefficients $\hat{c}_{\avec,\lambda}(q)$ are palindromic.
This is no longer the case in the circular setting.

\medskip 

\subsection{The double-complete graph}

Let $B_n$ denote the double-complete directed graph on $n$ vertices, that is, the graph with directed edges $i \to j$ for all $i \neq j$.
Consider the associated LLT polynomial, $\LLT_{B_n}(\xvec;q)$,
and define $H_{n}(\xvec;q) \coloneqq q^{\binom{n}{2}}\LLT_{B_n}(\xvec;q^{-1})$,
which is easier to work with in this case. We have that
\[
 H_n(\xvec;q) \coloneqq \sum_{F \in [n]^n} \xvec^w q^{m(F)}
\]
where $m(F)$ is the number of \emph{monochromatic edges}, when interpreting $w$ as a coloring of $K_n$.
In particular, if we let (compare with \cref{eq:lltRelation2})
\begin{align*}
\omega H_n(\xvec;q) = \sum_{\lambda \vdash n} \tilde{c}_{\lambda}(q) \left(\prod_{i=1}^{\length(\lambda)} \frac{(1-q)^{\lambda_i}}{1-q^{\lambda_i}} \right) \frac{\psumP_\lambda(\xvec)}{z_\lambda},
\end{align*}
then $\tilde{c}_{\lambda}(q) = q^{\binom{n}{n}} \hat{c}_{B_{n},\lambda}(q^{-1})$.
We now consider the generating function for $H_n(\xvec;q)$, and have that
\begin{align*}
F(\xvec;q) \coloneqq  \sum_{n \geq 0} \frac{H_n(\xvec;q)}{n!} &=\sum_n \sum_{\mu \vdash n}\frac{m_\mu(\xvec)}{\mu !} \cdot q^{\sum_j \binom{\mu_j}{2}}
 = \prod_{i} \left( \sum_{j\geq 0} q^{  \binom{j}{2} } \frac{x_i^{j}}{j!}  \right),
\end{align*}
because for each coloring with $\mu_i$ colors $i$, there are $\binom{\mu_i}{2}$ monochromatic edges. 
For each collection $\mu$ of labels, we then have $\binom{n}{\mu_1,\mu_2,\ldots}$ 
ways of placing the labels. 
Finally, dividing by $n!$ and splitting the sums independently over each variable we get the identity.
Next, to express in terms of power sum symmetric functions, define the expansion of the series above as
\begin{align}\label{eq:bng-def}
\log\left( \sum_{j \geq 0} q^{\binom{j}{2}} \frac{x^j}{j!} \right) \eqqcolon \sum_{r\geq 1} \frac{ g_r(q) x^r }{r!} = L(x;q),
\end{align}
to obtain
\begin{align*}
\log F(\xvec;q) = \sum_i \log\left( \sum_{j\geq 0} q^{  \binom{j}{2} } \frac{x_i^{j}}{j!}  \right) 
=\sum_i \sum_{r\geq 1} \frac{ g_r(q)x_i^r}{r!}=\sum_{r \geq 1} \frac{g_r(q) p_r(\xvec)}{r!}.
\end{align*}
Hence,
\[
\sum_{n \geq 0} \frac{H_n(\xvec;q)}{n!} = \exp \left( \sum_{r\geq 1} \frac{ g_r(q) p_r(\xvec) }{r!} \right).
\]
It follows that
\begin{align}
\tilde{c}_{(n)}(q) &= z_{(n)} \frac{1-q^n}{(1-q)^n} (-1)^{n-1} n! [p_n(\xvec)]  \sum_{n \geq 0} \frac{H_n(\xvec;q)}{n!} \notag \\
&= n \frac{q^n-1}{(q-1)^n} g_n(q) = \frac{n(q^n-1)}{(q-1)^n} \left[\frac{ x^n}{n!}\right] \log\left( \sum_{j \geq 0} q^{\binom{j}{2}} \frac{x^j}{j!} \right),
\label{eq:gnVsCtilde}
\end{align}
or equivalently as a generating function we have the following identity
\begin{align}\label{eq:cnGenFun}
\log \left(\sum_{j \geq 0} q^{\binom{j}{2}} \frac{x^j}{j!} \right)= \sum_{n \geq 1} \tilde{c}_{(n)}(q) \frac{(q-1)^n}{n(q^n-1)} \frac{x^n}{n!}.
\end{align}
\medskip 

The generating functions above implies the following recurrence for $\tilde{c}_{(m)}$:
\[
\frac{\tilde{c}_{(m)}(q)}{m}
=
\frac{q^m-1}{(q-1)^m}
\left[
q^{\binom{m}{2}} - \sum_{r=1}^{m-1} \binom{m-1}{r-1} \frac{ q^{\binom{m-r}{2}} (q-1)^{r} }{  q^{r}-1 } \frac{\tilde{c}_{(r)}(q)}{r}
\right],
\quad
\tilde{c}_{(1)}(q) =1.
\]
As an example, the coefficient $\tilde{c}_{(5)}(q)$ is
\[
 5 q^{10}+25 q^9+75 q^8+175 q^7+325 q^6+500 q^5+600 q^4+550 q^3+450 q^2+300 q+120.
\]

\medskip
We shall now connect the $\tilde{c}_{(m)}(q)$ with the theory of \emph{parking functions}.
Let $\PF(n) =\{ \avec=(a_1,\ldots,a_n): 1\leq \sort(\avec)_i \leq i, i=1,\dotsc,n\}$ be the
set of parking functions on $n$ cars, where the $i$th car has a preferred spot $a_i$,
and $\sort(\avec)$ is $\avec$ arranged in increasing order.
The graphical representation of parking functions is a lattice path $\gamma$ from $(0,0)$ to $(n,n)$,
such that there are $\#\{i: a_i=j\}$ vertical steps with $x$-coordinate $j-1$, and the corresponding indices $i :a_i=j$ are written in increasing order in the boxes to the right of these steps.
The parking function condition is equivalent to $\gamma$ being a Dyck path.
The area of a parking function is defined as the area of the corresponding Dyck path.
\begin{example}
As an example, $(1,3,4,1,1,4,1)$ is a parking function,
with the graphical representation
\begin{align}\label{eq:pfExample}
\begin{ytableau}
*(lightgray)  & *(lightgray)  & *(lightgray)  &  6   &  &  & *(yellow)  \\
*(lightgray)  & *(lightgray)  & *(lightgray)  &  3   &   &  *(yellow) \\
*(lightgray)  & *(lightgray)  & 2 &  & *(yellow) \\
  7  &   &    &  *(yellow)\\
  5  &  & *(yellow) \\
  4 & *(yellow) \\
   *(yellow) 1
\end{ytableau}
\end{align}
The area of the parking function is $13$.
\end{example}

\begin{theorem}
Let $\PF(n)$ be the set of all parking functions with $n$ letters.
Let $f_n(q) = \sum_{w \in \PF(n)} q^{\area(w)}$ be an associated $q$-weighted enumeration of parking functions.
We have the following relationship
\[
\tilde{c}_n(q) = n \frac{1-q^n}{1-q} f_{n}(q)q^{-n}.
\]
\end{theorem}
\begin{proof}
Let
\[
I_n(q) = \sum_{(a_1,\ldots,a_n) \in \PF(n)} q^{a_1+a_2+\cdots+a_n}
\] 
be the $q$-weight enumerator for parking functions.
It is easy to see that in the Dyck path representation
$a_1+\cdots+a_n +\area(\gamma) = \binom{n+1}{2}$, since $a_i$ is the $x$-coordinate of the label $i$ on the Dyck path and so $a_1+\cdots+a_n$ is the complementary area of the Dyck path inside the $(n\times n)$-box.
Hence $q^{\binom{n+1}{2}}I_n(q^{-1}) = f_n(q)$.
The following generating function has been derived in \cite{Kreweras1980}:
\[
\sum_{n\geq 1} q^{\binom{n}{2}}(q-1)^{n-1} I_n(q^{-1}) \frac{x^n}{n!} =\log \sum_{n\geq 0} q^{\binom{n}{2}}\frac{x^n}{n!}.
\]
The right hand side matches the generating function expansion in \eqref{eq:cnGenFun}, hence
\[
q^{\binom{n}{2}} (q-1)^{n-1}I_n(q^{-1}) = \frac{(q-1)^n}{n(q^n-1)} \tilde{c}_{(n)}(q), 
\]
and replacing $I_n(q^{-1})$ by $f_n(q) q^{-\binom{n}{2}-n}$ we get 
\[
\tilde{c}_{(n)}(q) = f_n(q) n \frac{q^n-1}{q-1}q^{-n}.
\]
\end{proof}
Note that $I_n(q+1) = \sum_{G} q^{e(G)-n}$, where $G$ runs over all simple connected graphs on $n$ vertices.

In the study of diagonal harmonics, a central operator on symmetric functions is
the $\nabla$-operator, for which the modified Macdonald polynomials are eigenfunctions.
The polynomial $f_n(q)$ is related to the $\nabla$-operator in the following sense:
The quasi-symmetric expansion of $\nabla \elementaryE_n$ can be expressed as
\[
 \nabla \elementaryE_n = \sum_{w \in \PF(n)} t^{\area(w)} q^{\dinv(w)} Q_{\ides(w)}.
\]
where $\dinv$ and $\ides$ are certain statistics on parking functions ---
see \cite{CarlssonMellit2015} for a recent proof of this identity (the ``shuffle'' conjecture), originally conjectured in \cite{HaglundHaimanLoehr2005}.

\subsection{Vertical strip case}

For general vertical strips, the relation in \cref{eq:lltRelation2} does not produce polynomial coefficients,
but computer experiments suggests the following conjectural generalization of the positivity in \cref{thm:lltPPositive}:
\begin{conjecture}\label{conj:powerSumPositivity}
Let $\nuvec$ determine a circular vertical strip digraph.
Then $\omega \LLT_\nuvec(\xvec;q+1)$ is $\psumP$-positive. 
Furthermore, the coefficients $c_{\nuvec,\lambda}(q)$
\begin{align*}
 \omega \LLT_\nuvec(\xvec;q+1) = \sum_{\lambda} c_{\nuvec,\lambda}(q) \frac{\psumP_\lambda(\xvec)}{z_\lambda}
\end{align*}
are polynomials with unimodal and non-negative integer coefficients.
\end{conjecture}

\section{Discussion on Schur positivity}\label{sec:ptableaux}

There is no known combinatorial proof of Schur positivity of 
vertical-strip LLT polynomials $\LLT_{\nuvec}(\xvec;q)$, not even in the case of unicellular diagrams.
However, there is a formula for the Schur expansion of $\chrom_\avec(\xvec;q)$ in the non-circular case
in terms of \emph{$P$-tableaux} appearing in \cite{Gasharov1996}.

For circular $\avec$, the polynomials $\LLT_\avec(\xvec;q)$ are not Schur-positive in general.
However, for circular vertical-strip LLT polynomials we conjecture that there is an expansion of the form
\begin{align}\label{eq:LLTSchurExpansion}
\LLT_\nuvec(\xvec;q+1) = \sum_{\theta \in O_\ast(\Gamma_\nuvec)} q^{\asc \theta} \sum_{\substack{ F : \Gamma_\nuvec \to [n] \\ F \text{ is } \theta-\text{compatible} \\
\text{+ extra condition}
 }}
\schurS_{\lambda(F)}(\xvec).
\end{align}
where the extra condition ensures to pick a highest weight representative for each Schur component.
The partition $\lambda(F)$ is given by $\lambda_i$ being the number of vertices with color $i$ in $F$,
and the extra condition should ensure that $\lambda(F)$ is indeed a partition.
This conjectured expansion is reminiscent of the Schur expansion of the chromatic quasisymmetric functions,
\cite{Gasharov1996}, which can be expressed in the following way:
\begin{align}
\chrom_\avec(\xvec;q) = \sum_{\theta \in AO(\Gamma_\avec)} q^{\asc \theta}
\sum_{\substack{ F : \Gamma_\avec \to [n] \\ F \text{ non-attacking} \\ F \text{ is } \theta-\text{compatible} \\
\text{$F$ is a $P$-tableau}
 }}
\schurS_{\lambda(F)}(\xvec).
\end{align}
Note that due to the non-attacking condition, each coloring appear for exactly one acyclic orientation,
so the above formula is expressed in a quite unnecessary manner --- we write it in this way to emphasize
the similarities with \cref{eq:LLTSchurExpansion}.

To give some additional support for the above expression, computer experiments
suggests the following property:
\begin{conjecture}
For a circular area sequence $\avec$, the difference
\[
\LLT_{\avec}(\xvec;q+1) - \chrom_\avec(\xvec;q)
\]
is Schur-positive.
\end{conjecture}
This conjecture suggests that colorings that are Gasharov's $P$-tableaux should be a subset of 
the colorings appearing in the sum in \eqref{eq:LLTSchurExpansion}.
This approach would be a new and unexplored avenue to give a combinatorial 
expansion of (vertical strip) LLT polynomials in the Schur basis.
The main difference compared to previous approaches is the $q+1$ shift and 
the fact that we know the generating $q$-statistic, rather than the combinatorial object to sum over.

\section{Linear relations among chromatic symmetric functions}

The following shows that every linear relation among a set of 
chromatic symmetric functions has a corresponding relation among LLT polynomials:
\begin{proposition}
Let $\avec^1,\dotsc,\avec^k$ be classical unit-interval graphs.
Then
\begin{align}
 \sum_{j=1}^k c_j(q) \chrom_{\avec^i}(\xvec;q) = 0 \text{ if and only if } \sum_{j=1}^k c_j(q) \LLT_{\avec^i}(\xvec;q) = 0.
\end{align}
for some $c_j(q)$.
\end{proposition}
\begin{proof}
This follows immediately from the plethystic
relation \cref{lem:plethRelation} between LLT polynomials and chromatic symmetric polynomials.
\end{proof}

\subsection{A principal specialization and Eulerian polynomials}

Given a symmetric function $f(\xvec)$, its \defin{principal specialization} is defined as $f(1,t,t^2,\dotsc)$.
One can show that the principal specialization of Schur polynomials is given by
\begin{align}
\schurS_\lambda(1,t,t^2,\dotsc) = \frac{ t^{n(\lambda)} }{\prod_{s \in \lambda} 1-t^{\mathrm{hook}(s)}}.
\end{align}
Moreover, notice that by the hook-length formula we have 
\[
(1-t)^n [n]!_t \schurS_\lambda(1,t,t^2,\ldots) \vert_{t=1} = f^{\lambda},
\]
the number of SYTs of shape $\lambda$, and equivalently, the coefficient of $x_1\cdots x_n$ 
in the expansion of $s_{\lambda}$. Since the Schur functions form a basis, 
this property directly extends to all symmetric functions: $(1-t)^n [n]!_t f(1,t,\ldots)|_{t=1}$ 
is the coefficient of the monomial $x_1\cdots x_n$ in $f$.


In the case of $\chrom_\avec$, the coefficient counts the number of all colorings of $\Gamma_\avec$ 
with distinct colors (hence all are proper) weighted by $q^{\asc F}$, in the 
case of line and cycle graphs the ascents are just the descents in the corresponding permutation. 
Thus we have the following:
\begin{proposition}[\cite{ShareshianWachs2014}]
Let $L_n$ be the area sequence determining a line graph on $n$ vertices. Then
\begin{align}
(1-t)^n [n]!_t \chrom_{L_n}(1,t,t^2,\dotsc;q) \big\vert_{t=1} = A_n(q)
\end{align}
where $A_n(q)$ is the Eulerian polynomial.
\end{proposition}

It directly extends to the cycle graph --- there are $n$ positions to put 
the number $n$, which always introduces one ascent. 
The remaining $n-1$ labels then form a permutation, where $q$ keeps track of the number of descents there. Hence we have
\begin{proposition}
Let $C_n$ be the area sequence determining a cycle graph on $n$ vertices. Then
\begin{align}
(t;t)_n \chrom_{C_n}(1,t,t^2,\dotsc;q) \big\vert_{t=1} = n q A_{n-1}(q)
\end{align}
where $A_n(q)$ is the Eulerian polynomial.
\end{proposition}
This is a special case of a more general theorem in~\cite{Ellzey2016}.

\section{Acyclic orientations and rook placements}

In this section, we give several combinatorial proofs for 
formulas concerning counting acyclic orientations of unit-interval graphs.
This gives alternative proofs for some identities 
given in \cite{ShareshianWachs2014,ShareshianWachs2011}.

We have seen that the area sequence, $\avec$, counting the number of inner shapes in each row determines
a unit interval graph. Similarly, the \defin{column area sequence} 
$\bvec = \{b_1,\dotsc,b_n\}$ is the list where $b_i$ counts the number
of squares in the inner shape in column $i$, from right to left.

\begin{lemma}\label{lemma:rowColumnBijection}
Suppose $\avec$ and $\bvec$ are the row and column area sequences, respectively, of a unit interval graph.
Then $\bvec$ is a permutation of $\avec$.
\end{lemma}
\begin{proof}
The proof is by induction on $n$, the number of vertices of the graph.
Consider the left-most column in the path, its peak is at $(1,b_n)$ and 
ends at row $i = n-b_n$. It must be at horizontal distance $b_n$ from the 
diagonal too, so $a_i = b_n$. We must have that $a_{i+j} = a_i -j$ for all $j \geq 0$ since these rows reach the end. 
Consider the Dyck path of height $n-1$ formed by removing the left-most column of the original path, so it has column sequence $(b_1,\ldots,b_{n-1})$. The row sequence is $a'=(a_1,\ldots,a_{i-1}, a_i-1,a_{i+1}-1,\ldots)=(a_1,\ldots,a_{i-1},a_{i+1},a_{i+2},\ldots,a_n)$. 
By induction, there is a permutation $\sigma$, s.t. $(b_1,\ldots,b_{n-1})=a'\circ \sigma$. 
Let $\phi(j) \coloneqq j$ if $j <i$ and $\phi(j) \coloneqq j+1$ if $j\geq i$, 
then $a'_j = a_{\phi(j)}$, so $b_j = a'_{\sigma(j)} = a_{\phi(\sigma(j))}$.
Finally, set $b_n = a_{i}$, so the permutation that sends $a$ to $b$ is $(\phi(\sigma),i)$. 
\end{proof}

As before, we represent acyclic orientations of $\Gamma_\avec$ by marking the 
inner squares with arrows pointing either right or down. 
Arrows pointing down represent ascending edges, $i\to j$ where $i<j$.
The number of ascending edges in an orientation $\theta$ is denoted $\asc(\theta)$.

\begin{proposition}\label{prop:uniqueAcyclic}
Let $\avec = \{a_1,\dotsc,a_n\}$ be a row area sequence
and $\{v_1,\dotsc,v_n\}$ be non-negative integers such that $v_i \leq a_i$.
Then there is a unique acyclic orientation of $\Gamma_\avec$ with $v_i$ ascending edges in row $i$.
\end{proposition}
\begin{proof}
We do proof by induction over the number of rows. The statement is trivial for one row.
Suppose there is already an acyclic orientation of rows $i+1,\dotsc,n$.
We restrict our attention to rows $i,i+1,\dotsc,i+a_i$. 
This cuts out a triangle as in \eqref{eq:triangle}.
\begin{align}\label{eq:triangle}
\begin{ytableau}
 *(lightgray) &  &    &   & *(yellow) i\\
  &   &   & *(yellow) \scriptstyle{i+1}\\
  &   & *(yellow) \scriptstyle{i+2}\\
    & *(yellow) \iddots \\
 *(yellow) 
\end{ytableau}
\end{align}
The vertices $i+1,\dotsc,i+a_i$ are totally ordered by the acyclic orientation in the corresponding rows.
Thus, there is a unique subset consisting of the $v_i$ maximal vertices among $i+1,\dotsc,i+a_i$
in this total order. Connect vertex $i$ to these via ascending edges and let the remaining edges in row $i$
be descending. 
It is clear from the construction that this is an acyclic orientation and that this is unique.
\end{proof}

\begin{corollary}[See {\cite[Thm 6.9]{ShareshianWachs2014}}]\label{cor:areaSeqOrientations}
Suppose $\chrom_\avec(\xvec;q) = \sum_\mu c_\mu(q) e_\mu$. Then
\[
 \sum_\mu c_\mu(q) = \prod_{i=1}^{n}  [a_i+1]_q = \prod_{i=1}^{n}  [b_i+1]_q
\]
where $b_1,\dotsc,b_n$ is the column area sequence of $\Gamma_\avec$.
\end{corollary}
\begin{proof}
The first equality follows from Proposition~\ref{prop:chromaticAcyclicOrientations} with $t=1$ 
and the above \cref{prop:uniqueAcyclic}. 
The second equality follows from the bijection in \cref{lemma:rowColumnBijection}.
\end{proof}

\begin{example}
Consider the diagram with area sequence $(2,2,3,2,1,0)$.
\[
\begin{ytableau}
*(lightgray) & *(lightgray) & *(lightgray) &    &  & *(yellow) \\
*(lightgray) & *(lightgray) &   &   & *(yellow)\\
  &  &   & *(yellow)\\
  &   & *(yellow)\\
  & *(yellow)\\
*(yellow)
\end{ytableau}
\]
We compute that
\begin{align*}
\chrom_{223210}(\xvec;q) &= (1+4 q+8 q^2+11 q^3+12 q^4+11 q^5+8 q^6+4 q^7+q^8) e_{11111} \\
  &+(q^2+3 q^3+4 q^4+3 q^5+q^6) e_{21110} 
\end{align*}
and verify that
\begin{align*}
&(1+4 q+8 q^2+11 q^3+12 q^4+11 q^5+8 q^6+4 q^7+q^8) \\
&+(q^2+3 q^3+4 q^4+3 q^5+q^6) \\
&= [1+1]_q[2+1]_q[3+1]_q[2+1]_q[2+1]_q.
\end{align*}
\end{example}

\begin{lemma}\label{lem:sinkInFirstRow}
Let $\Gamma_\avec$ be connected. The number of acyclic orientations of $\Gamma_\avec$ 
with one unique sink in the first row, where each ascending edge has weight $q$, is given by
\begin{equation}
\sum_{\substack{\theta \in AO(\Gamma_\avec)\\ \Sinks(\theta)=\{1\}  }} q^{\asc(\theta)}
= \prod_{i=1}^{n-1}  [a_i]_q = \prod_{i=1}^{n-1}  [b_i]_q.
\end{equation}
\end{lemma}
\begin{proof}
We will show that an acyclic orientation has a unique sink at vertex $1$ (first row) 
if and only if every column has at least one descending edge (\emph{i.e.} right arrow).

Consider an acyclic orientation with a unique sink at vertex $1$, and 
suppose there is a column with no descending edges, \emph{i.e.} only 
down arrows. Let this column be $i_1$, necessarily $i_1 >a_1$.
Since the unique sink is $1$, and $\Gamma_\avec$ is acyclic and connected,
there is a path $i_1 \to i_2 \to\cdots \to 1$. Since all arrows in column $i_1$ point down, 
we must have that $i_2>i_1$.   The only vertices connected to 1 are $2,\dotsc,a_1<i_1$, 
so there must be some $j$, for which $i_j > i_1 \geq i_{j+1}$. 
Then the row at $i_{j+1}$ must extend to column $i_j$, and thus 
intersect the $i_1$ column. Hence there is a down pointing arrow 
from $i_{j+1}$ to $i_1$, \emph{i.e.} $i_{j+1} \to i_1$, which 
creates a cycle $i_1 \to \cdots \to i_j \to i_{j+1} \to i_1$, leading to a contradiction.

Now suppose that we have an acyclic orientation with every column 
having at least one right arrow. Since it is acyclic, there must 
be at least one sink. No vertex with a nonempty column could be a sink 
because of the right arrow. Since $\Gamma_\avec$ is connected, 
the only empty column is $1$, so $1$ has to be the only sink.

By \cref{prop:uniqueAcyclic}, the orientation is uniquely specified by the number 
of down arrows in the columns, and every column can have at most $b_i-1$ down arrows. 
Hence there is a bijection with sequences $(v_1,\dotsc,v_n)$ with $v_i \in \{0,\dotsc,b_i-1\}$ 
with the total number of down arrows $v_1+\dotsb+v_n$, so 
\[
\sum_{\substack{\theta \in AO(\Gamma_\avec)\\ \Sinks(\theta)=\{1\}  }} q^{\asc(\theta)} =
\sum_{(v_1,\ldots,v_n)} q^{v_1+\cdots+v_n} = 
\prod_{i=1}^{n-1} \sum_{v_i=0}^{b_i-1} q^{v_i} = \prod_{i=1}^{n-1} [b_i]_q. \qedhere
\]
\end{proof}

\begin{proposition}[{\cite[Corr 7.2]{ShareshianWachs2014}}]
Suppose $\chrom_\pi(\xvec;q) = \sum_\mu c_\mu(q) e_\mu$. Then
\[
 c_{(n)} = [n]_q \prod_{i=1}^{n-1}  [a_i]_q.
\]
\end{proposition}
\begin{proof}
Remember that $c_{(n)}$ is the $q$-weighted count of acyclic orientations with one sink. 
Note that 
\[
 [n]_q \prod_{i=1}^{n-1}  [a_i]_q = q [a_1]_q \left( [n-1]_q [a_2]_q \dotsm [a_n]_q \right) +  [a_1]_q \dotsm [a_n]_q.
\]
The second term in the right hand side is the number of acyclic orientations with a unique sink in the first row,
according to \cref{lem:sinkInFirstRow}.
It suffices to show that 
\[
 q [a_1]_q \left( [n-1]_q [a_2]_q \dotsm [a_n]_q \right)
\]
counts the number of acyclic orientation with a unique sink not in the first row.
By induction, the expression in the parenthesis is the number of acyclic orientations with a unique sink somewhere in 
rows $2,\dotsc,n$. We can then augment this orientation with the first row by 
specifying the number of ascending arrows in that row.
In order to ensure that the first vertex is not a sink, there has 
to be at least one ascending arrow in the first row.
The weighted choice we can make here is therefore given by $q [a_1]_q$.
This completes the proof.
\end{proof}

\subsection{A connection with rook placements}

The formula in \cref{cor:areaSeqOrientations} appears in the study of rook placements 
and rook polynomials, see e.g. \cite{LewisMorales2016}.
In particular, it implies that acyclic orientations of a diagram with area sequence $\avec$,
is in bijection with $n$-rook placements on a Ferrers board with row lengths $r_i$ given by $a_i+i$.
This correspond to augmenting the triangular diagram such that it becomes a square:
\[
\begin{ytableau}
*(lightgray) & *(lightgray) & *(lightgray) &    &  & *(yellow) \\
*(lightgray) & *(lightgray) &   &   & *(yellow)\\
  &  &   & *(yellow)\\
  &   & *(yellow)\\
  & *(yellow)\\
*(yellow)
\end{ytableau}
\qquad
\longrightarrow 
\qquad
\begin{ytableau}
*(lightgray) & *(lightgray) & *(lightgray) &    &  & *(yellow) \\
*(lightgray) & *(lightgray) &   &   & *(yellow) & \\
  &  &   & *(yellow) & & \\
  &   & *(yellow) & & & \\
  & *(yellow) & & & &\\
*(yellow)  & & & & & \\
\end{ytableau}
\]

The $q$-weight of a rook placement is determined by the number of \emph{inversions}
in the rook placement. Given a rook placement, a square is considered an inversion
if it is part of the diagram, and it has no rooks above it in the same column,
and no rooks to the left in the same row.

In  \cref{prop:uniqueAcyclic}, we noted that the number of 
ascents in each row uniquely defines the acyclic orientation.
A similar property holds for rook placements, where the number of 
inversions in each row uniquely defines the rook placement, see \cite{GarsiaRemmel1986}.
By combining these two properties, we see that there is a unique bijection 
between acyclic orientations and rook placements that sends
ascending edges in row $i$ to inversions in row $i$ in the corresponding rook placement.

In \cref{eq:aorientRookBijExample}, we illustrate a rook placement 
where the bullets mark the inversions.
The left hand side is the corresponding acyclic orientation.
Note that each row contributes the same amount to the $q$-weight, 
and this property uniquely defines the bijection. 

\begin{align}\label{eq:aorientRookBijExample}
\begin{ytableau}
*(lightgray) & *(lightgray) & *(lightgray) & \downarrow    & \downarrow & *(yellow) \\
*(lightgray) & *(lightgray) & \downarrow  & \downarrow  & *(yellow)\\
 \rightarrow & \rightarrow & \downarrow  & *(yellow)\\
\rightarrow  & \rightarrow  & *(yellow)\\
\downarrow  & *(yellow)\\
*(yellow)
\end{ytableau}
\qquad
\longleftrightarrow
\qquad
\begin{ytableau}
*(lightgray) & *(lightgray) & *(lightgray) &  \bullet  & \bullet  & \times \\
*(lightgray) & *(lightgray) & \bullet  & \bullet  & \times & \\
 \bullet & \times  &   &   & & \\
\times  &   &  & & & \\
  &  & \bullet & \times & &\\
  & & \times & & & \\
\end{ytableau}
\end{align}

A bijection between rook placements and acyclic orientations is also given by A.~Hultman in 
the appendix of \cite{LewisMorales2016}, although this does not take the $q$-weight into account.

\subsection*{Acknowledgements}

The authors would like to thank Jim Haglund, Emily Sergel Leven, Alejandro Morales and Andy Wilson for
helpful discussions and Richard Stanley for general discussions and interesting questions. The authors also thank Brittney Ellzey and Michelle Wachs for help with the references.
The first author is funded by the \emph{Knut and Alice Wallenberg Foundation} (2013.03.07), 
the second author is partially supported by the NSF. 

\bibliographystyle{amsalpha}
\bibliography{bibliography}

\newcommand{\etalchar}[1]{$^{#1}$}
\providecommand{\bysame}{\leavevmode\hbox to3em{\hrulefill}\thinspace}
\providecommand{\MR}{\relax\ifhmode\unskip\space\fi MR }
\providecommand{\MRhref}[2]{%
  \href{http://www.ams.org/mathscinet-getitem?mr=#1}{#2}
}
\providecommand{\href}[2]{#2}
\begin{thebibliography}{HHL{\etalchar{+}}05b}

\bibitem[Ath15]{Athanasiadis15}
Christos~A. Athanasiadis, \emph{Power sum expansion of chromatic quasisymmetric
  functions}, Electronic Journal of Combinatorics \textbf{22} (2015), no.~2,
  1--9.

\bibitem[BM12]{Baur2012}
Karin Baur and Volodymyr Mazorchuk, \emph{{C}ombinatorial analogues of
  ad-nilpotent ideals for untwisted affine {L}ie algebras}, Journal of Algebra
  \textbf{372} (2012), 85--107.

\bibitem[CM15]{CarlssonMellit2015}
E.~{Carlsson} and A.~{Mellit}, \emph{{A proof of the shuffle conjecture}},
  ArXiv e-prints (2015).

\bibitem[DvW17]{DahlBergWilligenburg2017}
Samantha Dahlberg and Stephanie van Willigenburg, \emph{Lollipop and lariat
  symmetric functions}, ArXiv e-prints (2017).

\bibitem[Ell16]{Ellzey2016}
Brittney Ellzey, \emph{Chromatic quasisymmetric functions of directed graphs},
  ArXiv e-prints (2016).

\bibitem[Gas96]{Gasharov1996}
Vesselin Gasharov, \emph{Incomparability graphs of (3+1)-free posets are
  s-positive}, Discrete Mathematics \textbf{157} (1996), no.~1, 193--197.

\bibitem[GP13]{GuayPaquet2013}
Mathieu Guay-Paquet, \emph{{A modular relation for the chromatic symmetric
  functions of (3+1)-free posets}}, ArXiv e-prints (2013).

\bibitem[GP16]{GuayPaquet2016}
\bysame, \emph{{A second proof of the Shareshian--Wachs conjecture, by way of a
  new Hopf algebra}}, ArXiv e-prints (2016), 1--36.

\bibitem[GR86]{GarsiaRemmel1986}
A.M Garsia and J.B Remmel, \emph{{Q-counting rook configurations and a formula
  of Frobenius}}, Journal of Combinatorial Theory, Series A \textbf{41} (1986),
  no.~2, 246 -- 275.

\bibitem[Hag07]{qtCatalanBook}
James Haglund, \emph{The q,t-{C}atalan {N}umbers and the {S}pace of {D}iagonal
  {H}armonics ({U}niversity {L}ecture {S}eries)}, American Mathematical
  Society, 2007.

\bibitem[HHL05a]{Haglund2005Macdonald}
J.~Haglund, M.~Haiman, and N.~Loehr, \emph{{A combinatorial formula for
  Macdonald polynomials}}, J. Amer. Math. Soc. \textbf{18} (2005), no.~03,
  735--762.

\bibitem[HHL{\etalchar{+}}05b]{HaglundHaimanLoehr2005}
J.~Haglund, M.~Haiman, N.~Loehr, J.~B. Remmel, and A.~Ulyanov, \emph{A
  combinatorial formula for the character of the diagonal coinvariants}, Duke
  Mathematical Journal \textbf{126} (2005), no.~2, 195--232.

\bibitem[HMZ12]{Haglund2012compositionalShuffle}
J.~Haglund, J.~Morse, and M.~Zabrocki, \emph{A compositional shuffle conjecture
  specifying touch points of the {Dyck} path}, Canad. J. Math. \textbf{64}
  (2012), no.~4, 822--844.

\bibitem[HW17]{HaglundWilson2017}
James Haglund and Andrew~Timothy Wilson, \emph{Macdonald polynomials and
  chromatic quasisymmetric functions}, ArXiv e-prints (2017).

\bibitem[Kre80]{Kreweras1980}
G.~Kreweras, \emph{Une famille de polynômes ayant plusieurs propriétés
  énumeratives}, Periodica Mathematica Hungarica \textbf{11} (1980), no.~4,
  309--320.

\bibitem[LLT97]{Lascoux97ribbontableaux}
Alain Lascoux, Bernard Leclerc, and Jean-Yves Thibon, \emph{{R}ibbon
  {T}ableaux, {H}all-{L}ittlewood {F}unctions, {Q}uantum {A}ffine {A}lgebras
  {A}nd {U}nipotent {V}arieties}, J. Math. Phys \textbf{38} (1997), 1041--1068.

\bibitem[LM16]{LewisMorales2016}
Joel~Brewster Lewis and Alejandro~H. Morales, \emph{Combinatorics of diagrams
  of permutations}, Journal of Combinatorial Theory, Series A \textbf{137}
  (2016), 273 -- 306.

\bibitem[LT00]{LeclercThibon2000}
B.~Leclerc and J.-Y. Thibon, \emph{{Littlewood--Richardson coefficients and
  Kazhdan--Lusztig polynomials}}, Adv. Stud. Pure Math. \textbf{28} (2000),
  155--220.

\bibitem[Slo16]{OEIS}
N.~J.~A. Sloane, \emph{The on-line encyclopedia of integer sequences}, 2016.

\bibitem[Sta95]{Stanley95Chromatic}
Richard~P. Stanley, \emph{{A Symmetric Function Generalization of the Chromatic
  Polynomial of a Graph}}, Advances in Mathematics \textbf{111} (1995), no.~1,
  166--194.

\bibitem[Sta98]{Stanley98Chromatic}
\bysame, \emph{Graph colorings and related symmetric functions: ideas and
  applications}, Discrete Mathematics \textbf{193} (1998), no.~1, 267 -- 286.

\bibitem[Sta01]{StanleyEC2}
\bysame, \emph{{Enumerative Combinatorics: Volume 2}}, 1st ed., Cambridge
  University Press, 2001.

\bibitem[SW10]{ShareshianWachs2010}
John Shareshian and Michelle~L. Wachs, \emph{Eulerian quasisymmetric
  functions}, Advances in Mathematics \textbf{225} (2010), no.~6, 2921--2966.

\bibitem[SW11]{ShareshianWachs2011}
\bysame, \emph{{Chromatic quasisymmetric functions and Hessenberg varieties}},
  ArXiv e-prints (2011).

\bibitem[SW14]{ShareshianWachs2014}
\bysame, \emph{{Chromatic quasisymmetric functions}}, ArXiv e-prints (2014).

\end{thebibliography}

\end{document}